\documentclass[11pt,a4paper]{amsart}

\usepackage[margin=3cm]{geometry}

\usepackage{amsthm,amsmath,amssymb,amsfonts}
\usepackage{graphicx}
\usepackage{bm}
\usepackage{dsfont} 
\usepackage{mleftright}\mleftright
\usepackage{booktabs}

\numberwithin{equation}{section}
\numberwithin{table}{section}
\numberwithin{figure}{section}

\newtheorem{theorem}{Theorem}[section]
\newtheorem{lemma}[theorem]{Lemma}
\newtheorem{proposition}[theorem]{Proposition}
\newtheorem{corollary}[theorem]{Corollary}
\newtheorem{remark}[theorem]{Remark}
\newtheorem{definition}[theorem]{Definition}

\newcommand{\Ai}{\mathrm{Ai}}
\newcommand{\Bi}{\mathrm{Bi}}
\newcommand{\C}{\mathbb{C}}
\newcommand{\E}{\mathbb{E}}
\renewcommand{\Pr}{\mathbb{P}}
\newcommand{\R}{\mathbb{R}}
\newcommand{\1}{\mathds{1}}

\newcommand{\dd}{\mathrm d}
\newcommand{\ii}{\sqrt{-1}}

\newcommand\JJ{{\begin{pmatrix} 0 & 1 \\ -1 & 0 \end{pmatrix}}}

\renewcommand{\d}[2]{{d_{#1,#2}}}
\newcommand{\q}[2]{{q_{#1,#2}}}

\DeclareMathOperator{\diag}{diag}
\DeclareMathOperator{\pf}{pf}
\DeclareMathOperator{\sgn}{sgn}

\DeclareMathOperator{\tr}{tr}

\newcommand{\GL}{\mathrm{GL}}
\newcommand{\Herm}{\mathrm{Herm}}
\newcommand{\Sym}{\mathrm{Sym}}
\newcommand{\Skew}{\mathrm{Skew}}

\begin{document}

\title[]{Expected Euler characteristic method for the largest eigenvalue: \\ (Skew-)orthogonal polynomial approach}
\author[]{Satoshi Kuriki}
\address{The Institute of Statistical Mathematics, 10-3 Midoricho, Tachikawa, Tokyo 190-8562, Japan}
\email{kuriki@ism.ac.jp}

\begin{abstract}
The expected Euler characteristic (EEC) method is an integral-geometric method used to approximate the tail probability of the maximum of a random field on a manifold.
Noting that the largest eigenvalue of a real-symmetric or Hermitian matrix is the maximum of the quadratic form of a unit vector, we provide EEC approximation formulas for the tail probability of the largest eigenvalue of orthogonally invariant random matrices of a large class.
For this purpose, we propose a version of a skew-orthogonal polynomial by adding a side condition such that it is uniquely defined,
and describe the EEC formulas in terms of the (skew-)orthogonal polynomials.
In addition, for the classical random matrices (Gaussian, Wishart, and multivariate beta matrices), we analyze the limiting behavior of the EEC approximation as the matrix size goes to infinity under the so-called edge-asymptotic normalization.
It is shown that the limit of the EEC formula approximates well the Tracy-Widom distributions in the upper tail area, as does the EEC formula when the matrix size is finite.
\end{abstract}


\maketitle

\section{Introduction}

The set of $n\times n$ real symmetric or Hermitian matrices are denoted by
$\Sym(n)$ or $\Herm(n)$, respectively.
In this study, we deal with random matrices $A\in\Sym(n)$ or $\Herm(n)$ with probability density function with respect to $\dd A$ of the form
\begin{equation}
\label{pdf}
 p_n(A) \propto \prod_{i=1}^n w(\lambda_i),
\end{equation}
where $w(x)$ is a nonnegative function, $\lambda_1\le\cdots\le\lambda_n$ are the ordered eigenvalues of $A$, and
$\dd A$ is the Lebesgue measure of $\Sym(n)$ or $\Herm(n)$ identified with $\R^{n(n+1)/2}$ or $\R^{n^2}$.
It is assumed that $w(x)$ has a moment of arbitrary order.
The measure $\dd A$ is invariant under the orthogonal transformation
$\Sym(n)\ni A\mapsto P^\top A P$ ($P\in O(n)$) or $\Herm(n)\ni A\mapsto P^* A P$ ($P\in U(n)$),
where $P^\top$ and $P^*$ denote the transpose and the complex conjugate of $P$, respectively.
This implies that the distribution of $A$ is orthogonally invariant, and the eigenvalues and eigenvectors of $A$ are independently distributed.
The joint density of the eigenvalues with respect to $\prod_{i=1}^n \dd\lambda_i$ is
$\q{n}{\beta}(\lambda_1,\ldots,\lambda_n)\1_{\{\lambda_1<\cdots<\lambda_n\}}$,
where
\begin{equation}
\label{pdf-eigen}
 \q{n}{\beta}(\lambda_1,\ldots,\lambda_n) =
 \d{n}{\beta} \prod_{i=1}^n w(\lambda_i)\prod_{1\le i<j\le n} (\lambda_j-\lambda_i)^\beta
\end{equation}
with $\beta=1$ ($A\in\Sym(n)$) and $\beta=2$ ($A\in\Herm(n)$).
We do not need explicit forms of the normalizing constant $\d{n}{\beta}$.

The density function (\ref{pdf}) covers major orthogonally invariant random matrices: the standard symmetric Gaussian matrix (Gaussian orthogonal ensemble, GOE), the Wishart matrix, the multivariate beta matrix (MANOVA matrix), and their Hermitian counterparts (see (\ref{pdf-sym}) and (\ref{pdf-herm}) for the probability densities).
However, these are other random matrices with a probability density (\ref{pdf}).
For example, the positive definite conditional GOE matrix in Section \ref{subsec:pd_GOE} is included in this class.

In this paper, we discuss the approximation of the distribution of the largest eigenvalue $\lambda_{\max}(A)=\lambda_n(A)$ of $A$.
In statistics, the distribution of $\lambda_{\max}(A)$ appears as a null distribution of the test statistic in multivariate analysis.
For example, the largest eigenvalue of the Wishart matrix is used to test the independence in the correspondence analysis (\cite{haberman:1981},\cite{kateri:2014}), and the largest eigenvalue of the multivariate beta matrix is Roy's test statistic in multivariate analysis of variance (\cite{roy:1953},\cite{muirhead:1982}).
A complex Wishart matrix appears in the analysis of MIMO in wireless communication.
In this context, the upper tail probability corresponds to the $p$-value of the test statistic.
Because of this importance, numerical calculations or approximations of these distribution functions have been extensively studied in statistics.

One characterization of the largest eigenvalue is the maximum of a quadratic form:
\begin{equation}
\label{quadratic}
 \lambda_{\max}(A)=\max_{\Vert h\Vert=1} f(h), \quad \mbox{where}\ \ f(h)=h^\top A h \ \ \mbox{or}\ \ h^* A h.
\end{equation}
This is interpreted as the maximum of the random field $\{f(h)\}_{h\in M}$ on the unit sphere $M=\{h \mid \Vert h\Vert^2=1\}$.
To approximate the distribution of the maximum of such a random field on a manifold, an integral-geometric method, referred to as the expected Euler characteristic (EEC) method, was developed (\cite{adler:1981},\cite{worsley:1995},\cite{adler-taylor:2007}).
The set
\[
 M_x = \{h\in M \mid f(h)\ge x \} = f^{-1}([x,\infty))
\]
is referred to as the excursion set.
If the maximum of $f(h)$ is attained at a point with probability one,
when the threshold $x$ is large, we would expect that the set $M_x$ is a contractable set (homotopy equivalent to a point) unless it is empty.
This means that
\[
 \1(M_x\neq\emptyset) \approx \chi(M_x) \quad \mbox{when $x$ is large}.
\]
Then, by taking the expectations, we have
\begin{equation}
\label{eec}
 \Pr\biggl(\max_{h\in M}f(h)>x\biggr) \approx \E[\chi(M_x)] \quad
 \mbox{when $x$ is large}.
\end{equation}
This approximation is referred to as the expected Euler characteristic method.
\cite{takemura-kuriki:2002} pointed out that the volume-of-tube method having been developed for the same purpose is equivalent to the expected Euler characteristic method.
As we will see later in detail, in the particular case of (\ref{quadratic}),
the index set $M$ is redundant in the sense that
 $f(h)=f(-h)$ or $f(h)=f\bigl(e^{\ii\theta}h\bigr)$.
We need to redefine $M$ using the equivalence class $M:=M/\!\!\sim$ such that the maximum of $f(h)$ is attained at a point of the index set.

Based on this idea, \cite{kuriki-takemura:2001},\cite{kuriki-takemura:2008} demonstrated that the expected Euler characteristic method provides simple and accurate approximation formulas for the tail probabilities of the largest eigenvalues for real symmetric random matrices.
(See also \cite{takayama-etal:2020}, \cite{kuriki-takemura-taylor:2022}.)
The expected Euler characteristic method for the real Wishart matrix coincides with an approximation formula developed by \cite{hanumara-thompson:1968}, which is known in practice to be an accurate formula for the tail area.
The primary purpose of this study is to derive the expected Euler characteristics for random matrices with an arbitrary weight function $w(x)$.

For a real symmetric random matrix $A$, we show that the approximation formula includes the expected eigenpolynomial $\E[\det(x I_n -A)]$.
It is known that the skew-orthogonal polynomial is useful for evaluating the integral of the density function of real symmetric random matrices
(\cite{nagao-wadati:1991,adler-etal:2000,mehta:2004,ghosh:2009}).
The skew-orthogonal polynomial is not uniquely determined by the conventional definition.
In this paper, we propose an additional condition such that the skew-orthogonal polynomial with respect to arbitrary weight function $w(x)$ uniquely exists, and then provide a simple formula to express the expected eigenpolynomial using the proposed skew-orthogonal polynomial.
For the Hermitian matrix case,
the expectation of the squared eigenpolynomial $\E[\det(x I_n -A)^2]$ is needed,
which is known as a one-point correlation function in random matrix theory and is expressed through orthogonal polynomials.

For the major classical random matrices, the corresponding (skew-)orthogonal polynomials are known.
Hence, it is possible to express the expected Euler characteristics $\E[\chi(M_x)]$ in (\ref{eec}) explicitly.
This enables us to analyze the limiting behavior when the matrix size $n$ goes to infinity.
It is well-known that the limiting distribution of $\lambda_{\max}(A)$ when the matrix size $n$ goes to infinity is the Tracy-Widom distribution
 (\cite{tracy-widom:1994},\cite{tracy-widom:1996}).
We express the limiting formulas of the expected Euler characteristic using the Airy function, and show that it approximates accurately the upper tail probability of the Tracy-Widom distribution.
Although the Euler characteristic method provides an accurate approximation in upper tail area in general, this result is not obvious.
When the matrix size is infinite, the index set of the random field is inflated, and the regularity conditions of the expected Euler characteristic formula do not hold.

The remainder of this paper is constructed as follows.
In Section \ref{sec:ec}, we obtain the Euler characteristics of the quadratic random fields of real symmetric and Hermitian matrices.
In Section \ref{sec:ortho_poly}, we propose a version of a skew-orthogonal polynomial and provide the formulas for the expected Euler characteristics for a real symmetric matrix with an arbitrary weight function $w(x)$.
The corresponding results for Hermitian matrix are also given.
In Section \ref{sec:classical}, we obtain the expected Euler characteristic for several classical random matrices (including Gaussian, Wishart, multivariate beta, and conditional Gaussian matrices).
In Section \ref{sec:asymptotics}, we discuss the asymptotic behavior of the expected Euler characteristic when the sample size $n$ goes to infinity.
Some proofs are provided in Section \ref{sec:proofs}.
The basic integral formulas of the density functions, which are required in the proofs, are briefly summarized in Sections \ref{subsec:integral-herm} and \ref{subsec:integral-sym}.
In Appendix \ref{sec:painleve},
we analyze the exponential asymptotic structure
of the Painlev\'e II solution,
which is used in the analysis in Section \ref{sec:asymptotics}.

\section{Euler characteristics for quadratic-form random fields}
\label{sec:ec}

In this section, we formalize the expected Euler characteristics method for the random fields $f(h)$ defined in (\ref{quadratic}).
We summarize the results from \cite{kuriki-takemura:2008} first,
and extend them to the Hermitian case.
The proofs are summarized in Section \ref{sec:proofs}.

\subsection{Real symmetric case}
\label{subsec:real_symmetric}

Let $A$ be an $n\times n$ real symmetric random matrix with density $p_n$ in (\ref{pdf}).
Let $S(\R^{n\times 1})=\{h\in\R^{n\times 1} \mid h^\top h=1 \}$ be the set of unit column vectors in $\R^n$.
In addition, let
\[
 M = \{ h h^\top \mid h\in S(\R^{n\times 1}) \},
\]
the set of $n\times n$ real symmetric matrices of rank $1$.
$M$ is isomorphic to $S(\R^{n\times 1})/\!\sim$, where $h\sim -h$.
This quotient space is referred to as the real projective space $\mathbb{RP}^{n-1}$.
Then, the largest eigenvalue $\lambda_{\max}(A)$ of $A$ is the maximum of the random field
\[
 f(U) = \tr(U A), \ \ U\in M.
\]
The maximum of $f(U)$ is attained at a point in $M$ with probability one.
Let
\begin{equation}
\label{Mx}
 M_x = \{ U \in M \mid \tr(U A) \ge x\}
\end{equation}
be the excursion set, and let $\chi(M_x)$ be its Euler characteristic.
The Euler characteristic $\chi(M_x)$ is calculated using Morse's theorem:
\[
 \chi(M_x) = \sum_{\mbox{\footnotesize critical point}}\1(f(U^*)\ge x) \sgn\det(-\nabla\nabla f(U^*)),
\]
where $U^*$ is a critical point such that the gradient $\nabla f(U)$ vanishes,
and $\nabla\nabla f(U^*)$ is the Hesse matrix evaluated at the critical point.
This works if $f(U)$ is a Morse function.
\begin{lemma}[Real symmetric case \cite{kuriki-takemura:2008}]
\label{lem:chi-sym}
Suppose that $A$ is an $n\times n$ real symmetric random matrix with density $p_n$ in (\ref{pdf}).
With probability one,
\begin{equation}
\label{chiMx}
 \chi(M_x) = \sum_{k=1}^n (-1)^{n-k}\1(\lambda_k(A)\ge x).
\end{equation}
\end{lemma}
The details of the proof are summarized in Section \ref{subsec:morse}.
By letting $x\to -\infty$, we obtain a well-known formula $\chi(M)=\chi(\mathbb{RP}^{n-1})=(1-(-1)^n)/2$.

By taking the expectation for both sides of (\ref{chiMx}), we have the expected Euler characteristic
\begin{equation}
\label{EchiMx}
 \E\bigl[\chi(M_x)\bigr] = \sum_{k=1}^n (-1)^{n-k} \Pr(\lambda_k(A)\ge x).
\end{equation}
The expected Euler characteristic method uses it as an approximation formula to the upper tail probability of the largest eigenvalue:
\[
 \Pr(\lambda_{\max}(A)\ge x) \approx \E\bigl[\chi(M_x)\bigr] \quad\mbox{when $x$ is large}.
\]
Because $\Pr(\lambda_k(A)\ge x)$ increases in $k$,
$\Pr(\lambda_{\max}(A)\ge x) \ge \E\bigl[\chi(M_x)\bigr]$ holds
 (i.e., $\E\bigl[\chi(M_x)\bigr]$ is a liberal bound).

The expectation of the Euler characteristic is evaluated in the following lemma.
A proof is provided in Section \ref{subsec:morse}.
Although the Euler characteristic method provides a recipe for taking the expectation of $\chi(M_x)$ (which is the advantage of the expected Euler characteristic method), we prove it through direct calculations using the density function (\ref{pdf-eigen}) of the eigenvalues.
\begin{lemma}[Real symmetric case \cite{kuriki-takemura:2008}]
\label{lem:ec-sym}
\[
 \E[\chi(M_x)]
= \frac{\d{n}{1}}{\d{n-1}{1}} \int_x^\infty \E[\det(\lambda I_{n-1}-B)] w(\lambda) \dd\lambda,
\]
where $B$ is an $(n-1)\times (n-1)$ real symmetric random matrix with density $p_{n-1}$ in (\ref{pdf}),
and $\d{n}{1}$ is the normalizing constant in (\ref{pdf-eigen}) when $\beta=1$.
\end{lemma}

\subsection{Hermitian case}

Let $A$ be an $n\times n$ Hermitian random matrix with density $p_n$ in (\ref{pdf}).
Let
\[
 M = \{ h h^* \mid h\in S(\C^{n\times 1}) \},\quad S(\C^{n\times 1}) = \{ h\in\C^{n\times 1} \mid h^* h=1 \},
\]
be the set of $n\times n$ Hermitian matrices of complex-rank $1$.
The largest eigenvalue is the maximum of the random field
\[
 f(U) = \tr(U^* A), \ \ U\in M.
\]
Because $U=h h^*=(e^{i\theta}h) (e^{i\theta}h)^*$,
$M$ is isomorphic to $S(\C^{n\times 1})/\!\!\sim$, where $h\sim h'$ means $h'=e^{i\theta} h$ for some $\theta$.
This quotient space is referred to as the complex projective space $\mathbb{CP}^{n-1}$.

\begin{lemma}[Hermitian case]
\label{lem:chi-herm}
Suppose that $A$ is an $n\times n$ Hermitian random matrix with density $p_n$ in (\ref{pdf}).
With probability one,
\[
 \chi(M_x) = \sum_{k=1}^n \1(\lambda_k(A)\ge x).
\]
\end{lemma}
By letting $x\to -\infty$, we obtain a well-known formula $\chi(M)=\chi(\mathbb{CP}^{n-1})=n$.

The expected Euler characteristic method uses
\[
 \E\bigl[\chi(M_x)\bigr] = \sum_{k=1}^n \Pr(\lambda_k(A)\ge x)
\]
to approximate $\Pr(\lambda_{\max}(A)\ge x)$.
Clearly, $\Pr(\lambda_{\max}(A)\ge x)\le\E\bigl[\chi(M_x)\bigr]$
(i.e., $\E\bigl[\chi(M_x)\bigr]$ is a conservative bound).

A proof of the following lemma is provided in Section \ref{subsec:morse}.

\begin{lemma}[Hermitian case]
\label{lem:ec-herm}
\[
 \E[\chi(M_x)] = \frac{\d{n}{2}}{\d{n-1}{2}}
 \int_x^\infty \E[\det(\lambda I_{n-1}-B)^2] w(\lambda) \dd\lambda,
\]
where $B$ is an $(n-1)\times (n-1)$ Hermitian random matrix with density $p_{n-1}$ in (\ref{pdf}),
and $\d{n}{2}$ is the normalizing constant in (\ref{pdf-eigen}) when $\beta=2$.
\end{lemma}

\section{Expected Euler characteristic in terms of (skew-)orthogonal polynomials}
\label{sec:ortho_poly}

In the previous section, we observed that the expected Euler characteristic $\E[\chi(M_x)]$ has the integral forms of $\E[\det(x I -A)]$ or $\E[\det(x I -A)^2]$.
In this section, we evaluate them using the (skew-)orthogonal polynomials.

\subsection{Orthogonal and skew-orthogonal polynomials}

We introduce orthogonal and skew-orthogonal polynomials on $\R$ associated with the weight function $w(x)$.
It was assumed that $w(x)$ has a moment of arbitrary order.

First, we quickly review the orthogonal polynomials.
The orthogonal polynomial associated with the weight function $w(x)$ is a monic polynomial system $\phi_i(x)$, $i\ge 0$, such that $\deg\phi_i(x)=i$,
\begin{equation}
\label{ortho}
 \int_{\R} \phi_i(x) \phi_j(x) w(x) \dd x = h_i \delta_{ij}, \quad h_i>0,
\end{equation}
where $\delta_{ij}$ is the Kronecker delta.
Let
\[
 M_n = (m_{ij})_{0\le i,j\le n-1},\qquad m_{ij} = \int_{\R} x^{i+j} w(x) \dd x,
\]
be the moment matrix.
Under the assumptions on $w(x)$, $M_n$ is positive definite.
Hence, using the Gram-Schmidt normalization, there exists a unique $n\times n$ upper triangle matrix $T_n$ with unit diagonal elements such that $T_n^\top M_n T_n =\diag(h_i)_{0\le i\le n-1}$, where $h_i>0$.
The orthogonal polynomial system is then constructed as
\[
 (\phi_0(x),\ldots,\phi_{n-1}(x)) = (1,x,\ldots,x^{n-1}) T_n.
\]
The sequence $\phi_i(x)$, $i\ge 0$, is determined independently of the choice of $n$.

\begin{lemma}
\label{lem:norm-herm}
The normalizing constant in (\ref{pdf-eigen}) when $\beta=2$ is
\[
 1/\d{n}{2} = \prod_{i=0}^{n-1} h_i.
\]
\end{lemma}

The skew-orthogonal polynomial is introduced in a similar manner.
\begin{definition}
\label{def:skew-ortho}
A monic polynomial system $\varphi_i(x)$, $i\ge 0$, is said to be a skew-orthogonal polynomial associated with the weight function $w(x)$ if $\deg\varphi_i(x)=i$,
\begin{equation}
\label{cond1}
 \biggl(\int_{x<y}-\int_{x>y}\biggr) \varphi_i(x) \varphi_j(y) w(x) w(y) \dd x \dd y =
 \begin{cases}
  \sigma_k & (i,j)=(2k,2k+1), \\
 -\sigma_k & (i,j)=(2k+1,2k), \\
 0         & (\mbox{otherwise}),
 \end{cases}
\end{equation}
$\sigma_k>0$, $k\ge 0$, and
\begin{equation}
\label{cond2}
 \int_{\R} \varphi_i(x) w(x) \dd x =
 \begin{cases}
  \gamma_i>0 & (i:\mbox{even}), \\
  0          & (i:\mbox{odd}).
\end{cases}
\end{equation}
Define $\gamma_i$ for an odd $i$ by
\begin{equation}
\label{gamma-odd}
 \gamma_{2k+1}=\sigma_k/\gamma_{2k}>0, \ \ k\ge 0.
\end{equation}
\end{definition}

In the conventional definition (e.g., \cite{nagao-wadati:1991}), the skew-orthogonal polynomial does not require the condition (\ref{cond2}) and is not uniquely determined (Remark \ref{rem:not_unique}).
We will show the existence and uniqueness of the skew-orthogonal polynomial under Definition \ref{def:skew-ortho}.

Let
\begin{equation}
\label{skew-moment}
 S_n = (s_{ij})_{0\le i,j\le n-1},\quad
 s_{ij}=\biggl(\int_{x<y}-\int_{x>y}\biggr) x^i y^j w(x) w(y) \dd x \dd y,
\end{equation}
be the skew-moment matrix, and let
\begin{equation}
\label{moment}
 (\mu_0,\ldots,\mu_{n-1}),\quad
 \mu_i = \int_{\R} x^i w(x) \dd x,
\end{equation}
be the moment vector.
\begin{lemma}
\label{lem:USU}
There exists a unique $n\times n$ upper triangle matrix $U_n$ with unit diagonal elements such that
\begin{equation}
\label{cond1'}
 U_n^\top S_n U_n =
\begin{cases}
\displaystyle
 \diag\bigl(\sigma_0 J,\ldots,\sigma_{n/2-1} J\bigr) & (n:\mbox{even}), \\
\displaystyle
 \diag\bigl(\sigma_0 J,\ldots,\sigma_{(n-1)/2-1} J,0\bigr) & (n:\mbox{odd}),
\end{cases}
 \quad J = \begin{pmatrix} 0 & 1 \\ -1 & 0 \end{pmatrix}, \quad \sigma_k>0,
\end{equation}
and
\begin{equation}
\label{cond2'}
  (\mu_0,\ldots,\mu_{n-1}) U_n =
\begin{cases}
 (\gamma_0,0,\gamma_2,\ldots,\gamma_{n-2},0) & (n:\mbox{even}), \\
 (\gamma_0,0,\gamma_2,\ldots,0,\gamma_{n-1}) & (n:\mbox{odd}),
\end{cases}
 \quad \gamma_{2k}>0.
\end{equation}
\end{lemma}

The proof of Lemma \ref{lem:USU} is provided in Section \ref{sec:proofs}.
Once the matrix $U_n$ satisfying (\ref{cond1'}) and (\ref{cond2'}) is obtained, we immediately obtain the skew-orthogonal polynomial.

\begin{proposition}
The skew-orthogonal polynomial system $\varphi_i(x)$, $i\ge 0$, satisfying $\deg\varphi_i(x)=i$, (\ref{cond1}) and (\ref{cond2}) uniquely exists.
\end{proposition}

\begin{proof}
Let
\begin{equation}
\label{varphi}
 (\varphi_0(x),\ldots,\varphi_{n-1}(x)) = (1,x,x^2,\ldots,x^{n-1}) U_n.
\end{equation}
Here, $\varphi_i(x)$'s are defined independently of the choice of $n$.
\end{proof}

\begin{lemma}
\label{lem:norm-sym}
The normalizing constant in (\ref{pdf-eigen}) when $\beta=1$ is
\[
 1/\d{n}{1} =
 \begin{cases}
 \displaystyle
 \prod_{k=0}^{n/2-1}\sigma_k  & (n:\mbox{even}), \\
 \displaystyle
 \prod_{k=0}^{(n-1)/2-1}\sigma_k \times \gamma_{n-1} & (n:\mbox{odd}).
 \end{cases}
\]
\end{lemma}

\subsection{Expectation of the eigenpolynomial $\E[\det(x I_n-A)]$: Real symmetric case}

Now, we have the one of the main results.

\begin{lemma}
\label{lem:Edet}
Let $A_n$ be an $n\times n$ real symmetric random matrix with density $p_n$ in (\ref{pdf}).
Let $\varphi_i(x)$, $i\ge 0$, be a monic skew-orthogonal polynomial system associated with the weight function $w(x)$ satisfying (\ref{cond1}) and (\ref{cond2}).
Then,
\begin{equation*}
 \E[\det(z I_n-A_n)] = \widehat\varphi_n(x),
\end{equation*}
where
\begin{equation}
\label{hatvarphi}
 \widehat\varphi_n(x) =
\begin{cases}
 \varphi_n(x) & (n:\mbox{even}), \\
\displaystyle
 \varphi_n(x) + \frac{\gamma_n}{\gamma_{n-2}}\varphi_{n-2}(x)+\cdots+\frac{\gamma_n}{\gamma_{3}}\varphi_{3}(x)+\frac{\gamma_n}{\gamma_1}\varphi_1(x) & (n:\mbox{odd}),
\end{cases}
\end{equation}
in which $\gamma_i$ is defined in (\ref{gamma-odd}) in Definition \ref{def:skew-ortho}.

Conversely, the skew-orthogonal polynomial is expressed as
\[
 \varphi_n(x) = \begin{cases}
 \widehat\varphi_n(x) & (n:\mbox{even}), \\
\displaystyle
 \widehat\varphi_n(x) - \frac{\gamma_n}{\gamma_{n-2}} \widehat\varphi_{n-2}(x) & (n:\mbox{odd}).
\end{cases}
\]
\end{lemma}

\begin{remark}
\label{rem:not_unique}
Let $\{u_n\}$ be an arbitrary sequence.
The system
\[
 \varphi_n(x) := \begin{cases}
 \varphi_n(x) & (n:\mbox{even}), \\
 \varphi_n(x)-u_n \varphi_{n-1}(x) & (n:\mbox{odd})
 \end{cases}
\]
satisfies (\ref{cond1}), the conventional definition of the skew-orthogonal polynomial, but does not satisfy (\ref{hatvarphi}) unless $u_n\equiv 0$.
\end{remark}

By combining Lemmas \ref{lem:ec-sym}, \ref{lem:norm-sym}, and \ref{lem:Edet},
we have the expected Euler characteristic formula for real symmetric random matrices.

\begin{theorem}
\label{thm:eec-sym}
Let $A$ be an $n\times n$ real symmetric random matrix with density $p_n$ in (\ref{pdf}).
Let $M_x$ be the excursion set in (\ref{Mx}).
The expected Euler characteristic of $M_x$ is
\begin{equation}
\label{eec-sym}
 \E[\chi(M_x)] = \frac{1}{\gamma_{n-1}}
 \int_x^\infty \widehat\varphi_{n-1}(\lambda) w(\lambda) \dd\lambda,
\end{equation}
where $\widehat\varphi_{n-1}(x)$ is defined in (\ref{hatvarphi}), and $\gamma_{n-1}$ is defined in (\ref{cond2}) and (\ref{gamma-odd}).
\end{theorem}

The two lemmas below are crucial in proving Lemma \ref{lem:Edet}, and are of interest in themselves.

\begin{lemma}
\label{lem:monic}
Let $\varphi_i(x)$, $i\ge 0$, be a set of arbitrary monic polynomials of the indexed degree.
Suppose that
\begin{equation}
\label{3-term}
 x \varphi_i(x) = \varphi_{i+1}(x) + \sum_{j=0}^{i} t_{ij} \varphi_j(x).
\end{equation}
In matrix form, (\ref{3-term}) for $0\le i\le n-1$ is written as
\begin{equation}
\label{3-term-matrix}
 x \Phi_{n}(x) = L_n \Phi_{n}(x) + \bm{e}_n \varphi_n(x),
\end{equation}
where
\[
 \Phi_{n}(x) = \begin{pmatrix} \varphi_0(x) \\ \varphi_1(x) \\ \vdots \\ \varphi_{n-1}(x) \end{pmatrix}, \quad
 L_n =  \begin{pmatrix}
 t_{00} & 1      &        & 0 \\
 t_{10} & t_{11} & \ddots & \\
 \vdots & \vdots & \ddots & 1 \\
 t_{n-1,0} & t_{n-1,1} & \cdots & t_{n-1,n-1} \end{pmatrix}, \quad
 \bm{e}_n = \begin{pmatrix} 0 \\ \vdots \\ 0 \\ 1 \end{pmatrix}.
\]
Then,
\[
 \varphi_i(x) = \det(x I_i - L_i), \quad i\ge 1.
\]
\end{lemma}

\begin{proof}[Proof of Lemma \ref{lem:monic}]
It is easy to verify that $\det(x I_i - L_i)$ satisfies (\ref{3-term}).
\end{proof}

\begin{lemma}
\label{lem:even-odd}
Suppose that $A_n$ is an $n\times n$ real symmetric random matrix with density $p_n$ in (\ref{pdf}).
Then,
\begin{equation}
\label{even-odd}
 \int \E[\det(x I_n-A_n)] w(x) \dd x =
 \begin{cases}\displaystyle
  1/\gamma_{n} & (n:\mbox{even}), \\
  0            & (n:\mbox{odd}).
 \end{cases}
\end{equation}
\end{lemma}

\begin{proof}[Proof of Lemma \ref{lem:even-odd}]
Using the density function $\q{n}{1}(\lambda_1,\ldots,\lambda_n)$ in (\ref{pdf-eigen}),
the left-hand side of (\ref{even-odd}) is
\begin{align*}
& \int_{\R} w(x) \dd x \int_{\lambda_1<\cdots<\lambda_n} \prod_{i=1}^n (x-\lambda_i) \q{n}{1}(\lambda_1,\ldots,\lambda_n) \prod_{i=1}^n \dd\lambda_i \\
&= \d{n}{1} \int_{\R} w(x) \dd x \int_{\lambda_1<\cdots<\lambda_n} \prod_{i=1}^n (x-\lambda_i) \prod_{i=1}^n w(\lambda_i) \prod_{1\le i<j\le n}(\lambda_j-\lambda_i) \prod_{i=1}^n \dd\lambda_i \\
&= \d{n}{1} \sum_{k=0}^n \int_{\substack{\lambda_1<\cdots<\lambda_k<x, \\ x<\lambda_{k+1}<\cdots<\lambda_n}} w(x) \prod_{i=1}^n w(\lambda_i) \prod_{i=1}^n (x-\lambda_i) \prod_{1\le i<j\le n}(\lambda_j-\lambda_i) \,\dd x \prod_{i=1}^n \dd\lambda_i \\
&= \d{n}{1} \sum_{k=0}^n (-1)^{n-k} \times \int_{\lambda_1<\cdots<\lambda_n<\lambda_{n+1}} \prod_{i=1}^{n+1} w(\lambda_i)\prod_{1\le i\le j\le n+1} (\lambda_j-\lambda_i) \prod_{i=1}^{n+1} \dd\lambda_i \\
&= (\d{n}{1}/\d{n+1}{1}) \sum_{k=0}^n (-1)^{n-k} \times \int_{\lambda_1<\cdots<\lambda_n<\lambda_{n+1}} \q{n+1}{1}(\lambda_1,\ldots,\lambda_{n+1}) \prod_{i=1}^{n+1} \dd\lambda_i \\
&= \begin{cases} \d{n}{1}/\d{n+1}{1}=1/\gamma_n & (n:\mbox{even}), \\ 0 & (n:\mbox{odd}) \end{cases}
\end{align*}
by Lemma \ref{lem:norm-sym}.
\end{proof}

\subsection{Expectation of the squared eigenpolynomial $\E[\det(x I_n-A)^2]$: Hermitian case}

We move on to the Hermitian case.
Suppose that an $n\times n$ Hermitian random matrix $A$ is distributed with density $p_n$ in (\ref{pdf}).
The quantity $\E[\det(x I_n -A)^2] w(x)=\E[\prod_{i=1}^n(x-\lambda_i(A))^2] w(x)$ is proportional to the marginal density function of an unordered eigenvalue (referred to as the one-point correlation function) of the $(n+1)\times (n+1)$ Hermitian random matrix $A_{n+1}$ with density $p_{n+1}$ and has been well studied (e.g., \cite[Section 5.7]{mehta:2004}).
\begin{lemma}
\label{lem:Edet2}
Let $A_n$ be an $n\times n$ Hermitian random matrix with density $p_n$ in (\ref{pdf}).
Let $\phi_i(x)$, $i\ge 0$, be the monic orthogonal polynomial system associated with the weight function $w(x)$ satisfying (\ref{ortho}).
Then,
\[
 \E[\det(x I_n-A_n)^2] = \widehat\phi_n(x),
\]
where
\begin{equation}
\label{hatphi}
 \widehat\phi_n(x)
 = \phi_n(x)\phi'_{n+1}(x)-\phi_{n+1}(x)\phi'_n(x), \quad
 \phi'_i(x) = \frac{\dd\phi_i(x)}{\dd x}.
\end{equation}
\end{lemma}

By combining Lemmas \ref{lem:ec-herm}, \ref{lem:norm-herm}, and \ref{lem:Edet2},
we have the expected Euler characteristic formula for Hermitian random matrices.

\begin{theorem}
\label{thm:eec-herm}
Let $A$ be a Hermitian random matrix with density $p_n$ in (\ref{pdf}).
Let $M_x$ be the excursion set in (\ref{Mx}).
The expected Euler characteristic of $M_x$ is
\begin{equation}
\label{eec-herm}
 \E[\chi(M_x)] = \frac{1}{h_{n-1}}
 \int_x^\infty \widehat\phi_{n-1}(\lambda) w(\lambda) \dd\lambda,
\end{equation}
where $\widehat\phi_{n-1}(x)$ is defined in (\ref{hatphi}), and $h_{n-1}$ is the orthonormalizing constant (\ref{ortho}).
\end{theorem}

\section{Expected Euler characteristic for classical random matrices}
\label{sec:classical}

\subsection{Gaussian, Wishart, and multivariate beta matrices}
\label{subsec:classical}

Recall that the Gaussian orthogonal ensemble $\textsf{GOE}_n$,
the Wishart matrix $\textsf{W}_n(n+\alpha,I_n)$,
and the multivariate beta (MANOVA) matrix $\textsf{B}_n(n+\alpha,n+\beta,I_n)$
have densities $p_n(A)$ proportional to
\begin{equation}
\label{pdf-sym}
 e^{-\frac{1}{2}\tr A^2}, \quad
  \det(A)^{\frac{\alpha-1}{2}} e^{-\frac{1}{2}\tr A}\1_{\{A \succ 0\}}, \quad
 \det(A)^{\frac{\alpha-1}{2}} \det(I_n-A)^{\frac{\beta-1}{2}}\1_{\{0\prec A \prec I\}},
\end{equation}
respectively, where $A\succ B$ and $B\prec A$ indicate that $A-B$ is positive definite.
Their counterparts in the Hermitian matrices are
the Gaussian unitary ensemble $\textsf{GUE}_n$,
the complex Wishart matrix $\textsf{CW}_n(n+\alpha,I_n)$,
and the complex multivariate beta matrix $\textsf{CB}_n(n+\alpha,n+\beta,I_n)$
with the densities $p_n(A)$ proportional to
\begin{equation}
\label{pdf-herm}
 e^{-\tr A^* A}, \quad \det(A)^{\alpha} e^{-\tr A}\1_{\{A \succ 0\}}, \quad
 \det(A)^{\alpha} \det(I_n-A)^{\beta}\1_{\{0\prec A \prec I\}},
\end{equation}
respectively.
For these random matrices, everything is written down explicitly.
The weight function $w(x)$ is tabulated in Table \ref{tab:weights}.

\renewcommand{\arraystretch}{1.25}
\begin{table}[h]
\caption{Weight functions $w(x)$}
\label{tab:weights}
\begin{center}
\begin{small}
\begin{tabular}{ll}
\toprule
\hfil Distribution\hfil & \hfil $w(x)$\hfil \\
\midrule
$\textsf{GOE}_n$ & $e^{-x^2/2}$ \\
$\textsf{W}_n(n+\alpha,I_n)$ & $x^{\frac{\alpha-1}{2}}e^{-x/2}\1_{\{x>0\}}$ \\
$\textsf{B}_n(n+\alpha,n+\beta,I_n)$ & $x^{\frac{\alpha-1}{2}}(1-x)^{\frac{\beta-1}{2}}\1_{\{0<x<1\}}$ \\
\midrule
$\textsf{GUE}_n$ & $e^{-x^2}$ \\
$\textsf{CW}_n(n+\alpha,I_n)$ & $x^{\alpha}e^{-x}\1_{\{x>0\}}$ \\
$\textsf{CB}_n(n+\alpha,n+\beta,I_n)$ & $x^{\alpha}(1-x)^{\beta}\1_{\{0<x<1\}}$ \\
\bottomrule
\end{tabular}
\end{small}
\end{center}
\end{table}

The orthogonal polynomials $\phi_n(x)$, the orthonormalizing constants $h_n$ in (\ref{ortho}), and the derivatives of $\phi_n(x)$ are summarized in Table \ref{tab:ortho_poly}.
Here, $H_n(x)$, $L_n^{(\alpha)}(x)$, and $P_n^{(\alpha,\beta)}(x)$ are conventional Hermite, Laguerre, and Jacobi polynomials associated the weights
$e^{-x^2}$, $x^{\alpha}e^{-x}$, and $(1-x)^{\alpha}(1+x)^{\beta}$,
respectively \cite{szego:1975}.
The monic Hermite and Laguerre polynomials are denoted by $\bar H_n(x)$ and $\bar L_n^{(\alpha)}(x)$.
The monic polynomial of the shifted Jacobi polynomial $(-1)^n P_n^{(\alpha,\beta)}(1-2x)$ on the range $(0,1)$ is denoted by $\bar P_n^{(\alpha,\beta)}(x)$.
We have the concrete forms of the expected Euler characteristic $\E[\chi(M_x)]$ by substituting the quantities in Table \ref{tab:ortho_poly} into (\ref{eec-herm}) in Theorem \ref{thm:eec-herm}.

The skew-orthogonal polynomials $\varphi_n(x)$, the constants $\sigma_k$ in (\ref{cond1}) and $\gamma_n$ in (\ref{cond2}) and (\ref{gamma-odd}), and $\widehat\varphi_n(x)$ in (\ref{hatvarphi}) are listed in Table \ref{tab:skew-ortho_poly}.
The expressions for $\varphi_n(x)$ and $\sigma_k$ are found in \cite{nagao-wadati:1991,nagao-forrester:1995,adler-etal:2000}.
Although their definition does not require the condition (\ref{cond2}),
the formulas they provide satisfy (\ref{cond2}).
We have the concrete forms of the expected Euler characteristic $\E[\chi(M_x)]$ by substituting the quantities in Table \ref{tab:skew-ortho_poly} into (\ref{eec-sym}) in Theorem \ref{thm:eec-sym}.

Note that \cite{nagao:2007} provides explicit skew-orthogonal formulas other than those listed in Table \ref{tab:skew-ortho_poly}.

\renewcommand{\arraystretch}{1.25}
\begin{table}[h]
\caption{Monic orthogonal polynomials}
\label{tab:ortho_poly}
\begin{center}
\begin{small}
\begin{tabular}{@{}cl@{}cc@{}}
\toprule
$w(x)$ & \hfil $\phi_n(x)$ \hfil & $h_n$ in (\ref{ortho}) & $\phi'_n(x)$ \\
\midrule
$e^{-x^2}$ & $\bar H_n(x)=2^{-n}H_n(x)$ & $h_n=2^{-n}\sqrt{\pi}n!$ & $n\bar H_{n-1}(x)$ \\
$x^{\alpha}e^{-x}\1_{\{x>0\}}$ &
$\bar L^{(\alpha)}_n(x)=(-1)^n n! L^{(\alpha)}_n(x)$ & $h^{(\alpha)}_n = n!\Gamma(n+\alpha+1)$ & $n\bar L_{n-1}^{(\alpha+1)}(x)$ \\
$x^{\alpha}(1-x)^{\beta} \1_{\{0<x<1\}}$ &
$\begin{array}{l}
\!\!\bar P_n^{(\alpha,\beta)}(x) = (-1)^n \\
\times\frac{n!\,\Gamma(n+\alpha+\beta+1)}{\Gamma(2n+\alpha+\beta+1)} P_n^{(\alpha,\beta)}(1-2x)
\end{array}$ & $h_n^{(\alpha,\beta)}$ & $n\bar P_{n-1}^{(\alpha+1,\beta+1)}(x)$ \\
\bottomrule
\end{tabular}
\end{small}
\smallskip
\begin{small}
\[
 h_n^{(\alpha,\beta)} = \frac{n!\Gamma(n+\alpha+1)\Gamma(n+\beta+1)\Gamma(n+\alpha+\beta+1)}{\Gamma(2n+\alpha+\beta+1)\Gamma(2n+\alpha+\beta+2)}.
\]
\end{small}
\end{center}
\end{table}

\renewcommand{\arraystretch}{1.25}
\begin{table}[h]
\caption{Monic skew-orthogonal polynomials}
\label{tab:skew-ortho_poly}
\begin{center}
\begin{small}
\begin{tabular}{@{}cl@{}c@{}c@{}c@{}}
\toprule
$w(x)$ & \hfil $\varphi_{n}(x)$, $n=2k,2k+1$\hfil & $\sigma_k$ in (\ref{cond1}) & $\gamma_n$ in (\ref{cond2}), (\ref{gamma-odd}) & $\widehat\varphi_n(x)$ \\
\midrule
$e^{-x^2/2}$ &
$\left\{\!\!\begin{array}{l}
\bar H_{2k}(x) \\ \bar H_{2k+1}(x)-k\bar H_{2k-1}(x) \end{array}\right.$ &
$\begin{array}{l}\sigma_k= \\ 2^{-2k+1}\sqrt{\pi}(2k)!\end{array}$ &
$\gamma_n=\sqrt{2}\Gamma(\frac{n+1}{2})$ & $\bar H_n(x)$ \\
%
$x^{\frac{\alpha-1}{2}}e^{-x/2}\1_{\{x>0\}}$ &
$\left\{\!\!\begin{array}{l}
\bar L^{(\alpha)}_{2k}(x) \\ \bar L^{(\alpha)}_{2k+1}(x) -2k(2k+\alpha) \\ \qquad\qquad\qquad \times \bar L^{(\alpha)}_{2k-1}(x) \end{array}\right.$ &
$\begin{array}{l}
\sigma_k^{(\alpha)}= 4(2k)! \\ \times\Gamma(2k+\alpha+1)\end{array}$ &
$\begin{array}{l}
\gamma_n^{(\alpha)}=\frac{2^{n+\frac{\alpha+1}{2}}}{\sqrt{\pi}} \\
\times\Gamma(\frac{n+1}{2})\Gamma(\frac{n+\alpha+1}{2})
\end{array}$ & $\bar L_{n}^{(\alpha)}(x)$ \\
%
$\begin{array}{r}x^{\frac{\alpha-1}{2}}(1-x)^{\frac{\beta-1}{2}} \\ \times\1_{\{0<x<1\}}\end{array}$ &
$\left\{\!\!\begin{array}{l}
\bar P_{2k}^{(\alpha,\beta)}(x) \\
\bar P_{2k+1}^{(\alpha,\beta)}(x) - \frac{\gamma_{2k+1}^{(\alpha,\beta)}}{\gamma_{2k-1}^{(\alpha,\beta)}}\bar P_{2k-1}^{(\alpha,\beta)}(x) \end{array}\right.$ &
$\sigma_k^{(\alpha,\beta)}$ & $\gamma_n^{(\alpha,\beta)}$ & $\bar P_n^{(\alpha,\beta)}(x)$ \\
\bottomrule
\end{tabular}
\end{small}
\smallskip
\begin{small}
\begin{align*}
\sigma_k^{(\alpha,\beta)} & = \frac{4\Gamma(2k+1)\Gamma(2k+\alpha+1)\Gamma(2k+\beta+1)\Gamma(2k+\alpha+\beta+1)}{\Gamma(4k+\alpha+\beta+1)\Gamma(4k+\alpha+\beta+3)}, \\
\gamma_n^{(\alpha,\beta)} & =
\frac{2^{2n+\alpha+\beta}\Gamma(\frac{n+1}{2})\Gamma(\frac{n+\alpha+1}{2})\Gamma(\frac{n+\beta+1}{2})\Gamma(\frac{n+\alpha+\beta+1}{2})}{\pi \Gamma(2n+\alpha+\beta+1)}, \\
\frac{\gamma_{2k+1}^{(\alpha,\beta)}}{\gamma_{2k-1}^{(\alpha,\beta)}} & =
\frac{2k(2k+\alpha)(2k+\beta)(2k+\alpha+\beta)}{(4k+\alpha+\beta+2)(4k+\alpha+\beta+1)(4k+\alpha+\beta)(4k+\alpha+\beta-1)}.
\end{align*}
\end{small}
\end{center}
\end{table}

\begin{remark}
{\rm (i)} In these three cases, it happens that $\widehat\varphi_n(x)=\phi_n(x)$ holds.
That is,
\begin{equation}
\label{Edet}
\E[\det(x I_n-A)]
= \begin{cases}
 \bar H_n(x), & A\sim\textsf{GOE}_n, \\
 \bar L^{(\alpha,\beta)}_n(x), & A\sim\textsf{W}_n(n+\alpha,I_n), \\
 \bar P^{(\alpha,\beta)}_n(x), & A\sim\textsf{B}_n(n+\alpha,n+\beta,I_n).
 \end{cases}
\end{equation}

{\rm (ii)} In spite of the coincidence $\widehat\varphi_n(x)=\phi_n(x)$ above,
the expected eigenpolynomial $\E[\det(x I_n-A)]$ for a real symmetric matrix $A$ is not necessarily an orthogonal polynomial.
A counter example is presented in the next subsection (Proposition \ref{prop:counter_ex}).

{\rm (iii)} Contrary to the real symmetric case,
it is known that the expected eigenpolynomial $\E[\det(x I_n-A)]$ of a Hermitian random matrix $A$ is an orthogonal polynomial with respect to $w(x)$ (\cite{deift:1999}).

{\rm (iv)}
The relations (\ref{Edet}) can be proven directly.
For example, in the first case $A_n\sim\textsf{GOE}_n$, there exists $P_n\in O(n)$ such that
\[
 A_n =
 P_n \left(\begin{array}{cccc|c}
 a_{1} & b_{2}  &         &         & \\
 b_{2} & a_{2}  & \ddots  &         & \\
       & \ddots & \ddots  & b_{n-1} & \\
       &        & b_{n-1} & a_{n-1} & b_{n} \\
 \hline
       &        &         & b_{n}   & a_{n}
 \end{array}\right) P_n^\top,
\]
where $a_i \sim \textsf{N}(0,1)$ and $b_i\sim \chi_{i-1}/\sqrt{2}$ independently
(\cite{dumitriu-edelman:2002}).
Substituting this into $\det(x I_n - A_n)$, and expanding at the last row/column, we obtain the three-term relation
\[
 \E[\det(x I_n-A_n)] = \E[x-a_n]\E[\det(x I_{n-1}-A_{n-1})] - \E[b_n^2]\E[\det(x I_{n-2}-A_{n-2})],
\]
from which the first relation in (\ref{Edet}) follows.
This approach works for the other two cases.
\end{remark}

\begin{corollary}
As a corollary to Lemma \ref{lem:even-odd}, we have from (\ref{Edet}) that
\begin{align*}
& \int_{-\infty}^{\infty} \bar H_n(x) e^{-x^2/2} \dd x =
\begin{cases}
 1/\gamma_{n} & (n:\mbox{even}), \\
 0            & (n:\mbox{odd}),
\end{cases} \\
& \int_{0}^{\infty} \bar L^{(\alpha)}_n(x) x^{(\alpha-1)/2} e^{-x/2} \dd x =
\begin{cases}
 1/\gamma_n^{(\alpha)} & (n:\mbox{even}), \\
 0                     & (n:\mbox{odd}),
\end{cases} \\
& \int_{0}^{1} \bar P^{(\alpha,\beta)}_n(x) x^{(\alpha-1)/2}(1-x)^{(\beta-1)/2} \dd x =
\begin{cases}
 1/\gamma_n^{(\alpha,\beta)} & (n:\mbox{even}), \\
 0                           & (n:\mbox{odd}),
\end{cases}
\end{align*}
where
$\bar H_n(x)$, $\bar L^{(\alpha)}_n(x)$, and $\bar P^{(\alpha,\beta)}_n(x)$ are monic Hermite, Laguerre, and shifted Jacobi polynomials defined in Table \ref{tab:ortho_poly}, respectively, and
$\gamma_n$, $\gamma_n^{(\alpha)}$, and $\gamma_n^{(\alpha,\beta)}$ are listed in Table \ref{tab:skew-ortho_poly}.
\end{corollary}
These formulas were proven by \cite[Appendix A]{nagao-forrester:1995}
 by direct calculations.
It is not obvious that the integrals vanish when $n$ is odd.
The proof of Lemma \ref{lem:even-odd} reveals the universal structure behind.

\subsection{Positive definite conditional GOE}
\label{subsec:pd_GOE}

In this subsection, we deal with the random matrix $A\in\Sym(n)$ with density proportional to
\[
 e^{-\frac{1}{2}\tr A^2} \1_{\{A\succ 0\}} = \prod_{i=1}^n w(\lambda_i), \quad w(\lambda) = e^{-\frac{1}{2}\lambda^2}\1_{\{\lambda>0\}},
\]
where $\lambda_1\le\cdots\le\lambda_n$ are ordered eigenvalues of $A$.
The distribution of $A$ is $\textsf{GOE}_n$ under the condition that $A$ is positive definite, denoted by $A\sim\textsf{GOE}_n|_{\succ 0}$.
This distribution appears in the one-sided likelihood ratio test for the equality of two Wishart matrices \cite{kuriki:1993}.
The orthogonal polynomial associated with the weight function $e^{-\lambda^2/2}\1_{\{\lambda>0\}}$ is referred to as the half-range generalized Hermite polynomials, but the explicit expression is not known (\cite{ball:2002}).
Here, we derive the skew-orthogonal polynomial $\varphi_n(x)$ and $\widehat\varphi_n(x)=\E[\det(x I_n-A)]$, and the coefficients $\sigma_k$ and $\gamma_i$, at up to $n=3$ according to the procedure described in Section \ref{sec:ortho_poly}.

We first evaluate the moment $\mu_i$ in (\ref{moment}) and the skew-moment in (\ref{skew-moment}):
\[
 s_{ij} = \int_0^\infty x^i e^{-x^2/2} \dd x \biggl( \int_x^\infty - \int_0^x\biggr) y^j e^{-y^2/2} \dd y.
\]
The recursion formula below can be used to evaluate it:
\[
 s_{ij} = 2\int_0^\infty x^{i+j-1} e^{-x^2} \dd x -\delta_{j,1} \int_0^\infty x^{i} e^{-x^2/2} \dd x + (j-1) s_{i,j-2}, \quad s_{ii}=0, \quad s_{ji}=-s_{ij}.
\]

$\bullet$ $\varphi_i(x)$ and $\widehat\varphi_i(x)$ in (\ref{hatvarphi}):
\begin{align*}
\varphi_0(x) =& \widehat\varphi_0(x) = 1 \\
\varphi_1(x) =& \widehat\varphi_1(x) = x-\sqrt{\frac{2}{\pi }} \\
\varphi_2(x) =& \widehat\varphi_2(x) = x^2-\frac{\sqrt{2}+2}{\sqrt{\pi}}x+\frac{1}{\sqrt{2}} \\
\varphi_3(x) =&
 x^3-\frac{(\sqrt{2}-1)\sqrt{2\pi}}{\pi-2\sqrt{2}}x^2
+\frac{4(\sqrt{2}+3)-(3\sqrt{2}+1)\pi}{\sqrt{2}(\pi-2\sqrt{2})}x
-\frac{4(\sqrt{2}+1)-3\pi}{\sqrt{\pi}(\pi-2\sqrt{2})} \\
\widehat\varphi_3(x) =&
x^3-\frac{(\sqrt{2}-1)\sqrt{2\pi}}{\pi-2\sqrt{2}}x^2
+\frac{8\sqrt{2}-3\pi}{2(\pi-2\sqrt{2})}x
-\frac{(-2\sqrt{2}+3)\sqrt{\pi}}{\sqrt{2}(\pi-2\sqrt{2})}
\end{align*}

$\bullet$ $\sigma_k$ in (\ref{cond1}):
\[
 (\sigma_0,\sigma_1)=
 \Biggl(\frac{(\sqrt{2}-1)\sqrt{\pi}}{\sqrt{2}},\frac{(\sqrt{2}+3)(-4\sqrt{2}+7\pi-16)}{28\sqrt{\pi}}\Biggr)
\]

$\bullet$ $\gamma_i$ in (\ref{cond2}) and (\ref{gamma-odd}):
\[
 (\gamma_0,\gamma_1,\gamma_2,\gamma_3)
 = \Biggl(\sqrt{\frac{\pi}{2}},\sqrt{2}-1,\frac{(\sqrt{2}+1)(\pi-2\sqrt{2})}{2\sqrt{\pi}},\frac{(2\sqrt{2}-1)(-4\sqrt{2}-16+7\pi)}{14(\pi-2\sqrt{2})}\Biggr)
\]

Figure \ref{fig:cond_goe} shows the upper probabilities of the largest eigenvalue of positive definite GOE matrices ($\textsf{GOE}_n|_{\succ 0}$, $n=3,4$) and their EEC approximations provided by Theorem \ref{thm:eec-sym}.
The EEC approximation is liberal as we have seen in Section \ref{subsec:real_symmetric}.
The expected Euler characteristic method  precisely approximates the tail distribution of the largest eigenvalue.

\begin{figure}[h]
\begin{center}
\begin{tabular}{ccc}
\scalebox{0.70}{\includegraphics{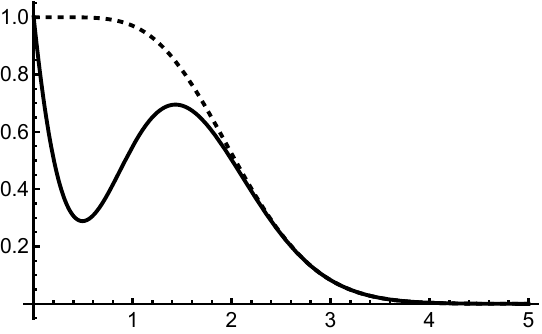}}
&\ \ &
\scalebox{0.70}{\includegraphics{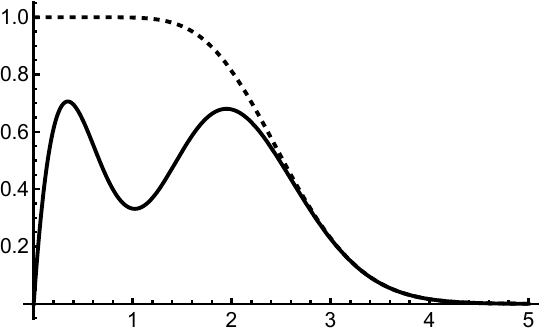}}
\\
{\small $\textsf{GOE}_3|_{\succ 0}$} && {\small $\textsf{GOE}_4|_{\succ 0}$}
\end{tabular}
\caption{Upper probabilities of the largest eigenvalue of positive definite GOE matrices.}
\label{fig:cond_goe}
\begin{small}
Dashed line: positive definite GOE, Solid line: EEC approximation.
\end{small}
\end{center}
\end{figure}

\begin{proposition}
\label{prop:counter_ex}
$\widehat\varphi_i(x)$, $i\ge 0$, cannot be an orthogonal polynomial under any weight function $w(x)$.
\end{proposition}

\begin{proof}
Suppose that such weight function $w(x)$ exists.
Let
\[
 \nu_k=\frac{\int x^k w(x) \dd x}{\int w(x) \dd x}.
\]
Then, the system
\[
 0=\frac{\int\widehat\varphi_i(x)\widehat\varphi_j(x) w(x) \dd x}{\int w(x) \dd x} \quad (0\le i<j,\ i+j\le 3)
\]
consists of four equations with three variables $\nu_1,\nu_2,\nu_3$.
It can be checked that the system has no solution.
\end{proof}

\section{Asymptotics when the matrix size approaches to infinity}
\label{sec:asymptotics}

When the matrix size approaches to infinity, the largest eigenvalue of a random matrix converges to a limit after a suitable normalization.
In this section we conduct an asymptotic analysis for the expected Euler characteristic.
For this purpose, the edge asymptotics for the Hermite, Laguerre, and Jacobi polynomials established by I.~Johnstone (\cite[Proposition 1]{johnstone-ma:2012},
\cite[Section 5]{johnstone:2001}, and
\cite[Proposition 2]{johnstone:2008}) are useful.
We summarize the results in modified forms which is convenient to our purpose.

\begin{proposition}[\cite{johnstone-ma:2012,johnstone:2001,johnstone:2008}]
\label{prop:johnstone}
Let $\Ai(x)$ and $\Ai'(x)$ be the Airy function of the first kind and its derivative function, respectively.

{\rm (i)} Hermite polynomial:
As $n\to\infty$,
\begin{equation}
\label{hermite-lim}
 2\frac{\sigma_n}{\gamma_n} w(x) \bar H_{n}(x) \Big|_{x=\mu_n+\sigma_n s}, \quad
 \mu_n = \sqrt{2 n},\ \ %
 \sigma_n = 2^{-1/2}n^{-1/6},
\end{equation}
and its derivative with respect to $s$ converge to $\Ai(s)$ and $\Ai'(s)$,
respectively, uniformly on any half interval $[s_0,\infty)$,
where $w(x)=e^{-x^2/2}$, and $\gamma_n$ is as listed in Table \ref{tab:skew-ortho_poly}.

{\rm (ii)} Laguerre polynomial:
As $n,\alpha\to\infty$ with $\alpha/n\to \bar\alpha$,
\begin{equation}
 2\frac{\sigma_n}{\gamma_n^{(\alpha)}} w^{(\alpha)}(x) \bar L_{n}^{(\alpha)}(x)\Big|_{x=\mu_n+\sigma_n s}, \quad
 \mu_n = (1+\sqrt{1+\bar\alpha})^2 n, \ \ %
 \sigma_n = \frac{(1+\sqrt{1+\bar\alpha})^{4/3}}{(1+\bar\alpha)^{1/6}}n^{1/3},
\label{laguerre-lim}
\end{equation}
and its derivative with respect to $s$ converge to $\Ai(s)$ and $\Ai'(s)$,
respectively, uniformly on any half interval $[s_0,\infty)$,
where $w^{(\alpha)}(x)=x^{(\alpha-1)/2}e^{-x/2}$, and $\gamma_n^{(\alpha)}$ is as listed in Table \ref{tab:skew-ortho_poly}.

{\rm (iii)} Jacobi polynomial:
As $n,\alpha,\beta\to\infty$ with $\alpha/n\to\bar\alpha$, $\beta/n\to\bar\beta$,
\begin{equation}
\label{jacobi-lim-1}
 2\frac{\sigma_n}{\gamma_n^{(\alpha,\beta)}} w^{(\alpha,\beta)}(x) \bar P_n^{(\alpha,\beta)}(x)\Big|_{x=\mu_n+\sigma_n s},
\end{equation}
where
\begin{equation}
\label{jacobi-lim-2}
 \mu_n = \frac{1-\cos(\varphi+\gamma)}{2}, \ \ %
 \sigma_n = \frac{1}{2}\biggl(\frac{2 \sin^{4}(\varphi+\gamma)}{(2+\bar\alpha+\bar\beta)^2\sin\varphi\sin\gamma}\biggr)^{1/3}n^{-2/3},
\end{equation}
with $\cos\varphi=(-\bar\alpha+\bar\beta)/(2+\bar\alpha+\bar\beta)$,
$\cos\gamma=(\bar\alpha+\bar\beta)/(2+\bar\alpha+\bar\beta)$,
and the derivative of (\ref{jacobi-lim-1}) with respect to $s$ converge to $\Ai(s)$ and $\Ai'(s)$,
respectively, uniformly on any half interval on $[s_0,1)$ ($s_0>0$),
where $w^{(\alpha,\beta)}(x)=x^{(\alpha-1)/2}(1-x)^{(\beta-1)/2}$, and $\gamma_n^{(\alpha,\beta)}$ is as listed in Table \ref{tab:skew-ortho_poly}.
\end{proposition}

Applying these properties orthogonal polynomials,
we evaluate the limiting behavior of the expected Euler characteristic method.
A proof is provided in Section \ref{sec:proofs}.

\begin{theorem}
\label{thm:ec-limit}
Let $\E[\chi(M_x)]$ be the expected Euler characteristic of the excursion set (\ref{Mx}), where
$A$ is a real symmetric random matrix $A\sim\textsf{GOE}_n$, $\textsf{W}_n(n+\alpha,I_n)$, $\textsf{B}_n(n+\alpha,n+\beta,I_n)$, or
a Hermitian random matrix $A\sim\textsf{GUE}_n$, $\textsf{CW}_n(n+\alpha,I_n)$, $\textsf{CB}_n(n+\alpha,n+\beta,I_n)$.
Let $\mu_n$ and $\sigma_n$ be defined in (\ref{hermite-lim}) for GOE/GUE, (\ref{laguerre-lim}) for a real/complex Wishart, and (\ref{jacobi-lim-2}) for real/complex multivariate beta, respectively.
As $n\to\infty$ with $\alpha/n\to\bar\alpha$, $\beta/n\to\bar\beta$,
\[
 \E[\chi(M_x)]\Big|_{x=\mu_{n-1}+\sigma_{n-1} s} \to
 \widehat F_1(s) = \frac{1}{2}\int_{s}^\infty\Ai(x) \dd x
\]
when $A$ is real symmetric, and
\[
 \E[\chi(M_x)]\Big|_{x=\mu_{n-1}+\sigma_{n-1} s} \to
 \widehat F_2(s) = \int_{s}^\infty \bigl(\Ai'(x)^2-x\Ai(x)^2\bigr) \dd x
\]
when $A$ is Hermitian.
\end{theorem}

\begin{remark}
Using the same scaling factors $\mu_n$ and $\sigma_n$,
it holds that as $n\to\infty$,
\begin{align*}
\Pr(\lambda_n(A)> x) \Big|_{x=\mu_n+\sigma_n s} \to F_{\textsf{TW}_1}(s) &\quad (\mbox{real symmetric}), \\
\Pr(\lambda_n(A)> x) \Big|_{x=\mu_n+\sigma_n s} \to F_{\textsf{TW}_2}(s) &\quad (\mbox{Hermitian}),
\end{align*}
where $F_{\textsf{TW}_1}$ and $F_{\textsf{TW}_2}$ are the upper probability functions of the Tracy-Widom distribution of the first and second kinds.
This normalization is referred to as edge-asymptotics.
\end{remark}

Figure \ref{fig:finite} shows the upper probabilities of the largest eigenvalue of the real Wishart $\textsf{W}_3(4,I)$ and complex Wishart $\textsf{CW}_3(4,I)$, and their EEC approximations.
The EEC approximation is conservative for the complex Wishart, and liberal for the real Wishart.
The remarkable advantage of the expected Euler characteristic method is its accuracy.
The magnitude of the relative error
\[
 \Delta_n(x) = \frac{\E[\chi(M_x)]-\Pr(\lambda_n(A)>x)}{\Pr(\lambda_n(A)>x)}
\]
is in general exponentially smaller than $\Pr(\lambda_n(A)>x)$ when $x$ is large.
See \cite{kuriki-takemura:2008} for details.

\begin{figure}[h]
\begin{center}
\begin{tabular}{ccc}
\scalebox{0.70}{\includegraphics{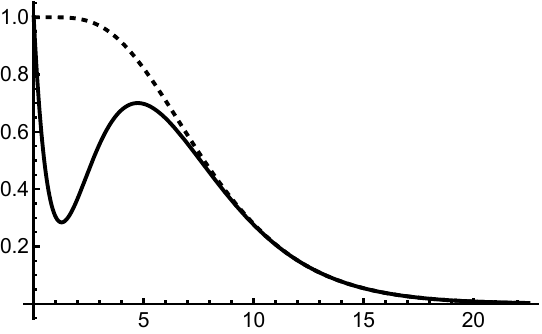}}
&\ \ &
\scalebox{0.70}{\includegraphics{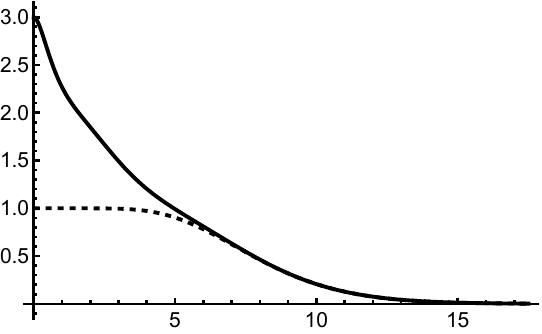}}
\\
{\small Wishart $\textsf{W}_3(4,I)$} && {\small Complex Wishart $\textsf{CW}_3(4,I)$}
\end{tabular}
\caption{Upper probabilities of the largest eigenvalue ($n=3$).}
\label{fig:finite}
\begin{small}
Dashed line: Wishart, Solid line: EEC approximation.
\end{small}
\end{center}
\end{figure}

The theorem below indicates that the tail accuracy is inherited when $n$ goes to infinity.
\begin{theorem}
\label{thm:airy-tw}
Define the limiting relative error to the Tracy-Widom distribution as follows:
\[
 \Delta_\beta(s) = \frac{\widehat F_\beta(s)-F_{\textsf{TW}_\beta}(s)}{F_{\textsf{TW}_\beta}(s)}, \quad \beta=1,2.
\]
Then,
\[
 \Delta_\beta(s) \sim \begin{cases}
\displaystyle
 - \frac{1}{2^{5}\pi^{1/2}}\cdot s^{-\frac{9}{4}}e^{-\frac{2}{3}s^{3/2}} & (\beta=1:\mbox{real symmetric case}), \\[3mm]
\displaystyle
   \frac{1}{2^{10}\pi}\cdot s^{-\frac{9}{2}}e^{-\frac{4}{3}s^{3/2}} & (\beta=2:\mbox{Hermitian case}),
\end{cases}
\]
as $s\to\infty$.
It holds that $\Delta_2(s)\sim\Delta_1(s)^2$.
\end{theorem}

\begin{proof}
In Theorem \ref{thm:exp_asympt}, we evaluate the order of the terms in the exponential asymptotic expansion of the solution of the Painlev\'e II differential equation.
Here, $F_{\textsf{TW}_1}$ and $F_{\textsf{TW}_2}$ are explicitly written based on the solution \cite{tracy-widom:1996}.
On the other hand, the asymptotic expansions of the Airy function $\Ai(x)$ and its derivative $\Ai'(x)$ that make up $\widehat F_1(s)$ and $\widehat F_2(s)$ are well-known \cite[Section 10.4]{abramowitz-stegun:1970}. 
\end{proof}

Figure \ref{fig:infinite} shows the upper probabilities of Tracy-Widom distributions and the limits of the EEC approximation.
The limits of the EEC formulas still approximate the true distributions in the limit.

\begin{figure}[h]
\begin{center}
\begin{tabular}{ccc}
\scalebox{0.75}{\includegraphics{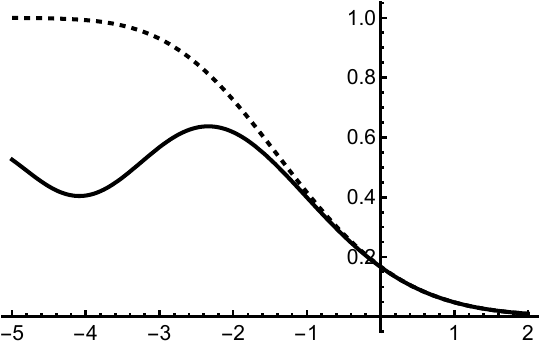}}
&\ \ &
\scalebox{0.75}{\includegraphics{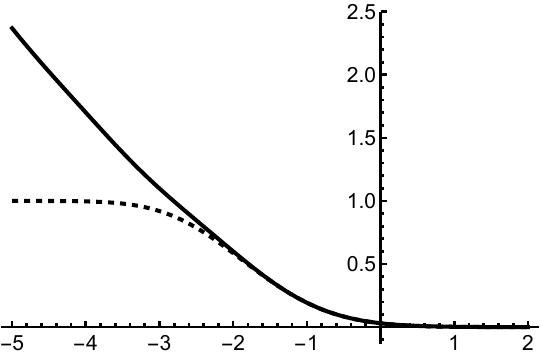}}
\\
{\small Real symmetric} && {\small Hermitian}
\end{tabular}
\caption{Upper probability of the largest eigenvalue when $n\to\infty$.}
\label{fig:infinite}
\begin{small}
Dashed line: Tracy-Widom, Solid line: Limit of EEC approximation.
\end{small}
\end{center}
\end{figure}

In this case, the expected Euler characteristic method behaves well as the matrix size $n$ goes to infinity.
However, this is not obvious, because $n\to\infty$ means that the dimension of the manifold $M$ goes to infinity, and the Morse theory used to calculate the Euler characteristic does not make sense.
The expected Euler characteristic method in an infinite dimensional space will be a challenging research topic in the future.

\section{Proofs}
\label{sec:proofs}

\subsection{Proofs of Lemmas \ref{lem:chi-sym}, \ref{lem:ec-sym}, \ref{lem:chi-herm} and \ref{lem:ec-herm}}
\label{subsec:morse}

\begin{proof}[Proof of Lemma \ref{lem:chi-sym}]
To analyze the behavior of $f$, we introduce the local coordinates $h=h(t)\in S(\R^{n\times 1})$, $t=(t_1,\ldots,t_{n-1})$.
For example, for a fixed point $h_0\in S(\R^{n\times 1})$, let $H_0$ be an $n\times n$ orthogonal matrix whose first column vector is $h_0$.
Then,
\[
\textstyle
 h(t)=H_0 \Bigl(\sqrt{1-\sum_{i=1}^{n-1} t_i^2},t_1,\ldots,t_{n-1}\Bigr)^\top
\]
is a local coordinate system of $ S(\R^{n\times 1})$ around $h_0$.
In addition, $h(t)h(t)^\top$ is a local coordinate of $M$ around $h_0 h_0^\top$.
Let $\partial_i=\partial/\partial t_i$, $h_i=\partial_i h$, and $h_{ij}=\partial_i\partial_j h$.
By taking the derivative of $h^\top h=1$, we have $h^\top h_i=0$ and $h^\top h_{ij}+h_i^\top h_j=0$.
Then, $\partial_i f=2 h^\top A h_i$, and
$\partial_i|_{t^*} f=0, \forall i$ ($t=t^*$ is a critical point)
$\iff$ $(A -\lambda I)h|_{t^*}=0, \exists\lambda$ ($h$ is an eigenvector).
In addition, $\partial_i\partial_j f=2 h_i^\top A h_j+2 h^\top A h_{ij}$,
and at a critical point $t=t^*$,
$\partial_i\partial_j|_{t^*} f=2(h_i^\top A h_j+h^\top A h_{ij})=2(h_i^\top A h_j-\lambda h_i^\top h_j)$.

When $h$ is an eigenvector of $A$, $A$ is written as
\[
 A = \lambda h h^\top + H B H^\top,
\]
where $H=H(h)$ is an $n\times (n-1)$ matrix-valued function of $h$ such that $(h,H)\in O(n)$.
Note that $\lambda = h^\top A h$ and $B=H^\top A H$.
We write $(h_1,\ldots,h_{n-1})=H T$, where $T\in\GL(n-1)$.
Then, $\partial_i\partial_j|_{t^*} f=2 T^\top(H^\top A H-(h^\top A h) I)T$.
According to the Morse theory,
if $f(U)$ is a Morse function,
\begin{align*}
 \chi(M_x)
&= \sum_{\mbox{\footnotesize critical point}} \1(f(U)\ge x)\,\sgn\det\bigl( -\partial_i\partial_j f \bigr) \\
&= \sum_{h\,:\,\mbox{\footnotesize eigenvector}} \1(h^\top A h \ge x)\,\sgn\det\Bigl((h^\top A h) I_{n-1} - H^\top A H \Bigr) \\
&= \sum_{k=1}^n \1(\lambda_k(A) \ge x)\,\sgn\prod_{l\ne k}(\lambda_k(A)-\lambda_l(A)) \\
&= \sum_{k=1}^n (-1)^{n-k}\1(\lambda_k(A)\ge x),
\end{align*}
where $\lambda_k(\cdot)$ is the $k$th smallest eigenvalue.
Here, $f(U)$ is a Morse function with probability one iff $A$ has distinct $n$ eigenvalues, which holds if $A$ has a density function with respect to the Lebesgue measure $\dd A$ of $\Sym(n)$.
\end{proof}

\begin{proof}[Proof of Lemma \ref{lem:chi-herm}]
We translate the problem of $\C^n$ into $\R^{2n}$.
The Hermitian matrix $A=B+\ii C\in\Herm(n)$ is represented as
\[
 A = \begin{pmatrix} B & -C \\ C & B \end{pmatrix} = \begin{pmatrix} B & C^\top \\ C & B \end{pmatrix} \in\Sym(2n).
\]
Let
\[
 f(U) = \tr(U A)/2, \quad U \in M = \{ \phi(h) \mid h\in S(\R^{2n\times 1}) \},
\]
where
\[
 \phi(h)=h h^\top + J h h^\top J^\top, \quad
 J = \begin{pmatrix} 0 & -I_n \\ I_n & 0 \end{pmatrix}.
\]
Here, $M$ is a class of the orthogonal projection matrix of rank $2$ in $\R^{2n}$.
Let
\[
 R(\theta) =
 \begin{pmatrix}
  \cos\theta I_n & -\sin\theta I_n \\ \sin\theta I_n & \cos\theta I_n
 \end{pmatrix}.
\]
By direct calculations, we see that
$f(\phi(h)A) = \tr(\phi(h) A)/2 = h^\top A h$, $R(\theta)^\top A R(\theta) = A$,
and $R(\theta)^\top \phi(h) R(\theta) = \phi(h)$ holds.

The dimension of $M$ is $2(n-1)$.
We introduce local coordinates $t=\bigl(t_1,\ldots,t_{2(n-1)}\bigr)$ on $M$.
For a fixed point $\phi(h_0) \in M$, $h_0 =(p_0^\top,q_0^\top)^\top$,
let $P_0+\ii Q_0$ be an $n\times n$ unitary matrix whose first column vector is $p_0+\ii q_0$, and let
\[
 h(t) = H_0 \begin{pmatrix} v(t) \\ w(t) \end{pmatrix}, \quad
 H_0 =
 \begin{pmatrix}
 P_0 & -Q_0 \\
 Q_0 &  P_0
\end{pmatrix} \in O(2n),
\]
where
\[
 \textstyle
 v(t)=\Bigl(\sqrt{1-\sum_{i=1}^{2(n-1)} t_i^2},t_1,\ldots,t_{n-1}\Bigr)^\top, \quad
 w(t)=\bigl(0,t_{n},\ldots,t_{2(n-1)}\bigr)^\top.
\]
Then, $\phi(h(t))\in M$, $h(0)=h_0$.
Moreover, because
\[
 E_i = \frac{\partial}{\partial t_i} \phi(h(t))\Big|_{t=0}, \quad 1\le i\le 2(n-1),
\]
are linearly independent, this parameterization is a local coordinate system of $M$ around $h_0$.

Let $H=H(h)$ be a $2n\times(n-1)$ matrix-valued smooth function of $h$ that satisfies $(h,J h,H,J H)\in O(2n)$.
For $h=(v^\top,w^\top)^\top$,
such $H$ is given as $H = (V^\top,W^\top)^\top$, where
$V,W\in\R^{n\times(n-1)}$ are matrices such that
$(v+\ii w,V+\ii W)$ is $n\times n$ unitary.

Here, $U=\phi(h)\in M$ is a critical point of $f(U) = h^\top A h$ iff
$\partial_i f=h_i^\top A h=0$ for $1\le i\le 2(n-1)$.
Moreover, because $AJ$ is skew-symmetric, $(J h)^\top A h=0$.
Note that $h$ is orthogonal to $J h$ and $h_i$, $1\le i\le 2(n-1)$,
and $\{J h,h_1,\ldots,h_{2(n-1)}\}$ are linearly independent around $h_0$
(because $J h$ and $h_i$'s are orthogonal at $h_0$).
Thus, at a critical point, we have
$(A-\lambda I)h=0,\,\exists\lambda$, that is, $h$ is an eigenvector of $A$.

Let
\[
 \bigl(h_1,\ldots,h_{2(n-1)}\bigr) = (H,J H) T + J h e^\top,
\]
where $T\in\GL(2(n-1))$ and $e$ is a $2(n-1)\times 1$ vector.
When $h$ is an eigenvector of $A$, $Jh$ is an eigenvector with the same eigenvalue, and $A$ is written as follows:
\[
 A = \lambda (h h^\top + J h h^\top J) + (H,J H) D \begin{pmatrix} H^\top \\ H^\top J^\top \end{pmatrix}.
\]

Hence, at a critical point $t=t^*$,
\[
 \partial_i\partial_j|_{t^*}f
 = h_i^\top A h_j + h^\top A h_{ij}
 = h_i^\top A h_j + \lambda h^\top h_{ij}
 = h_i^\top A h_j - (h^\top A h) h_i^\top h_j.
\]
Because
\[
 (h_i^\top h_j)
 = T^\top \begin{pmatrix} H^\top \\ H^\top J^\top \end{pmatrix} (H,J H) T + e e^\top
 = T^\top T + e e^\top
\]
and
\[
 (h_i^\top A h_j)
 = T^\top \begin{pmatrix} H^\top \\ H^\top J^\top \end{pmatrix} A (H,J H) T + e e^\top (h^\top A h),
\]
the Hesse matrix $(-\partial_i\partial_j f)$ at a critical point is
\[
 T^\top \Biggl((h^\top A h)I_{2(n-1)}-\begin{pmatrix} H^\top \\ H^\top J^\top \end{pmatrix}A(H,J H)\Biggr) T.
\]

Therefore, by Morse's theorem,
\begin{align*}
 \chi(M_x)
 &= \sum_{\mbox{\footnotesize critical point}} \1(f(U)\ge x)\,\sgn\det\bigl( -\partial_i\partial_j f \bigr) \\
 &= \sum_{h\,:\,\mbox{\footnotesize eigenvector}} \1(h^\top A h \ge x)\,\sgn\det\Biggl((h^\top A h)I_{2(n-1)}-\begin{pmatrix} H^\top \\ H^\top J^\top \end{pmatrix}A(H,J H)\Biggr) \\
 &= \sum_{k=1}^n \1(\lambda_k\ge x)\,\sgn\det\Biggl(\lambda_k I_{2(n-1)}-\begin{pmatrix} \Lambda_{(k)} & 0 \\ 0 & \Lambda_{(k)}\end{pmatrix}\Biggr), \qquad \Lambda_{(k)}= \diag(\lambda_i)_{1\le i\le n, i\ne k} \\
 &= \sum_{k=1}^n \1(\lambda_k\ge x)\sgn\prod_{i\ne k}(\lambda_k-\lambda_i)^2 \\
 &= \sum_{k=1}^n \1(\lambda_k\ge x),
\end{align*}
where $\lambda_k$ is the $k$th smallest eigenvalue of $A=B+\ii C$.
\end{proof}

\begin{proof}[Proof of Lemmas \ref{lem:ec-sym} and \ref{lem:ec-herm}]
We first prove Lemma \ref{lem:ec-sym}.
Using the joint density function $\q{n}{1}(\lambda_1,\ldots,\lambda_n)$ of the ordered eigenvalues in (\ref{pdf-eigen}),
\begin{align*}
& \frac{\dd}{\dd x} \Pr(\lambda_k >x) \\
&= \frac{\dd}{\dd x}\int_{x}^\infty \dd \lambda_k \int_{\substack{\lambda_1<\cdots<\lambda_{k-1}<\lambda_k, \\ \lambda_k<\lambda_{k+1}<\cdots<\lambda_n}} \q{n}{1}(\lambda_1,\ldots,\lambda_n) \dd \lambda_1\cdots \dd \lambda_{k-1} \dd \lambda_{k+1}\cdots \dd \lambda_n \\
&= -\int_{\substack{\lambda_1<\cdots<\lambda_{k-1}<x, \\ x<\lambda_{k+1}<\cdots<\lambda_n}} \q{n}{1}(\lambda_1,\ldots,\lambda_{k-1},x,\lambda_{k+1},\ldots,\lambda_{n}) \dd \lambda_1\cdots \dd \lambda_{k-1} \dd \lambda_{k+1}\cdots \dd \lambda_n \\
&= -\frac{\d{n}{1}}{\d{n-1}{1}}\int_{\lambda_1<\cdots<\lambda_n} w(x) \1_{\{\lambda_{k-1}<x<\lambda_k\}} (-1)^{n-k}\prod_{k=1}^{n-1}(x-\lambda_k) \\
&\qquad\times
 \q{n-1}{1}(\lambda_1,\ldots,\lambda_{n-1}) \dd \lambda_1\cdots \dd \lambda_{n-1}.
\end{align*}
Noting that $\sum_{k=1}^n \1_{\{\lambda_{k-1}<x<\lambda_k\}} =1$ a.e., we have
\begin{align*}
& \sum_{k=1}^n (-1)^{n-k} \frac{\dd}{\dd x} \Pr(\lambda_k>x) \\
&=  -\frac{\d{n}{1}}{\d{n-1}{1}}\int_{\lambda_1<\cdots<\lambda_n} w(x) \prod_{k=1}^{n-1}(x-\lambda_k) \q{n-1}{1}(\lambda_1,\ldots,\lambda_{n-1}) \dd \lambda_1\cdots \dd \lambda_{n-1} \\
&= -\frac{\d{n}{1}}{\d{n-1}{1}}\E[\det(x I_{n-1}-B)] w(x),
\end{align*}
which is the differentiation of $\E[\chi(M_x)]$ in (\ref{EchiMx}).
Combined with the boundary condition $\lim_{x\to\infty}\E[\chi(M_x)]=0$, we complete the proof.

The proof for Lemma \ref{lem:ec-herm} is almost the same except for the factor $(-1)^{n-k}$.
\end{proof}

\subsection{An integral formula for eigenvalue densities: Hermitian case}
\label{subsec:integral-herm}

The basic integration tool for managing the squared linkage factor $\prod_{i<j}(\lambda_j-\lambda_i)^2$ is as follows.
\begin{proposition}[e.g., {\cite[Section 5.7-5.8]{mehta:2004}}]
Let $\varphi_{i-1}(\lambda)$ be an arbitrary monic polynomial of degree $i-1$.
Let $g(\lambda)$ be an arbitrary function.
Then,
\[
 \int_{\lambda_1<\cdots<\lambda_n}
 \prod_{i=1}^n g(\lambda_i) \prod_{1\le i<j\le n} (\lambda_j-\lambda_i)^2 \prod_{i=1}^n \dd \lambda_i = \det(M),
\]
where
\[
 M = \biggl(\int \varphi_{i-1}(x)\varphi_{j-1}(x) g(x) \dd x\biggr)_{1\le i,j\le n}.
\]
\end{proposition}

This proposition follows from the Binet-Cauchy formula for a product of two rectangular matrices with the Vandermonde determinant (\ref{vandermonde}).

\begin{proof}[Proof of Lemma \ref{lem:norm-herm}]

We apply Proposition \ref{prop:pf-int} with $g(\lambda)=w(\lambda)$ and $\varphi_i(\lambda)$ as the monic orthogonal polynomial with respect to the weight function $w(\lambda)$.
Then, $M = \diag(h_0,\ldots,h_{n-1})$ and
\[
 1/\d{n}{2} = \det(M) = \prod_{i=0}^{n-1}h_k.
\]
\end{proof}

\subsection{Integral formula for eigenvalue densities: Real symmetric case}
\label{subsec:integral-sym}

Suppose that $A$ is an $n\times n$ real symmetric random matrix with probability density $p_n$ in (\ref{pdf}).
The joint density of the eigenvalues of $A$ is
$\q{n}{1}(\lambda_1,\ldots,\lambda_n)\1_{\{\lambda_1<\cdots<\lambda_n\}}$, where
$\q{n}{1}(\lambda_1,\ldots,\lambda_n)$ is given in (\ref{pdf-eigen}).
The factor $\prod_{i<j}(\lambda_j-\lambda_i)$ included in $\q{n}{1}(\lambda_1,\ldots,\lambda_n)$ is called the linkage factor.
The basic integration tool used to manage the linkage factor is as follows.
The set of $n\times n$ real skew-symmetric matrices is denoted by $\Skew(n)$.
\begin{proposition}[e.g., {\cite[Theorem 6.1]{baik-rains:2001},\cite[Section 5.5]{mehta:2004}}]
\label{prop:pf-int}
Let $\varphi_{i-1}(\lambda)$ be an arbitrary monic polynomial of degree $i-1$.
Let $g(\lambda)$ be an arbitrary function.
Define an $n\times n$ skew-symmetric matrix
\[
 S = \biggl(\biggl(\int_{x<y}-\int_{x>y}\biggr) \varphi_{i-1}(x)\varphi_{j-1}(y) g(x) g(y) \dd x \dd y \biggr)_{1\le i,j\le n}\in \Skew(n),
\]
and an $n\times 1$ vector
\[
 \bm{s} = \biggl(\int \varphi_{i-1}(y) g(y) \dd y \biggr)_{1\le i\le n}\in\R^{n\times 1}.
\]
Then,
\[
 \int_{\lambda_1<\cdots<\lambda_n}
 \prod_{i=1}^n g(\lambda_i) \prod_{1\le i<j\le n} (\lambda_j-\lambda_i) \prod_{i=1}^n \dd \lambda_i
 = \begin{cases}
 \pf(S) & (n:\mbox{even}), \\
 \pf\begin{pmatrix} S & \bm{s} \\ -\bm{s}^\top & 0 \end{pmatrix}
        & (n:\mbox{odd}),
  \end{cases}
\]
where $\pf(\cdot)$ is the Pfaffian of a skew-symmetric matrix.
\end{proposition}

This proposition follows from the Pfaffian analogue of the Binet-Cauchy formula (\cite{ishikawa-wakayama:1995}) and the fact that the linkage factor is the Vandermonde determinant:
\begin{equation}
\label{vandermonde}
 \prod_{1\le i<j\le n} (\lambda_j-\lambda_i)
 = \det(\lambda_j^{i-1})_{1\le i,j\le n} = \det(\varphi_{i-1}(\lambda_j))_{1\le i,j\le n}.
\end{equation}

Recall that the Pfaffian for an $n\times n$ skew-symmetric matrix $S=(s_{ij})_{1\le i,j\le n}$ is defined by the following:
\[
 \pf(S) = \begin{cases}
 \displaystyle
 \sum_{\pi\in\Pi_n}\sgn(\pi)s_{\pi(1),\pi(2)}\cdots s_{\pi(n-1),\pi(n)} & (n:\mbox{even}), \\
 0 & (n:\mbox{odd}),
\end{cases}
\]
where $\Pi_n$ is the set of permutations $\pi$ of $\{1,\ldots,n\}$ such that
\[
 \pi(1)<\pi(2), \ldots, \pi(n-1)<\pi(n), \ \pi(1)<\pi(3)<\cdots<\pi(n-1),
\]
i.e., the set of all pairings of $\{1,\ldots,n\}$.
For example,
\[
 \{(\pi(1),\pi(2),\pi(3),\pi(4)) \mid \pi \in \Pi_4\} = \{(1,2,3,4), (1,3,2,4), (1,4,2,3) \}.
\]
For a skew-symmetric matrix $S$ and a matrix $U$, the identities
\[
 \pf(S)^2 = \det(S), \qquad \pf(U^\top S U) = \pf(S)\det(U)
\]
hold.
The latter assures that a Gaussian elimination (simultaneously from the left and the right) keeps the Pfaffian invariant.
That is, for a skew-symmetric block matrix $(S_{ij})_{1\le i,j\le 2}$ with $S_{11}$ non-singuler,
\begin{equation}
\label{elimination}
\begin{aligned}
\pf\left(\begin{pmatrix} I & 0 \\ -S_{21}S_{11}^{-1} & I \end{pmatrix}
         \begin{pmatrix} S_{11} & S_{12} \\ S_{21} & S_{22} \end{pmatrix}
         \begin{pmatrix} I & -S_{11}^{-1}S_{12} \\ 0 & I \end{pmatrix}\right)
 =& \pf\begin{pmatrix} S_{11} & 0 \\ 0 & S_{22\cdot 1} \end{pmatrix} \\
 =& \pf(S_{11})\pf(S_{22\cdot 1}),
\end{aligned}
\end{equation}
where $S_{22\cdot 1}=S_{22}-S_{21}S_{11}^{-1}S_{12}$.

\subsection{Proofs of Lemmas \ref{lem:USU}, \ref{lem:norm-sym}, and \ref{lem:Edet}}

\begin{proof}[Proof of Lemma \ref{lem:USU}]

Recall first that
\[
 \pf(S_n) = \int\cdots\int_{x_1<\cdots<x_n} \prod_{1\le i<j\le n}(x_j-x_i)\prod_{i=1}^n w(x_i)\prod_{i=1}^n \dd x_i >0, \quad n=2,4,\ldots,
\]
meaning that the upper-left $i\times i$ minor of $S_n$ is non-singular for $i=2,4,\ldots$
Then, by applying $2\times 2$ block-wise Gaussian elimination to $S_n$, we find an $n\times n$ upper triangle matrix $\widetilde U_n$ of the form of
\begin{equation}
\label{upper_triangle}
\widetilde U_n=
\begin{pmatrix}
 I_2 & *   & \cdots & * \\
     & I_2 & \ddots & \vdots \\
     &     & \ddots & * \\
 0   &     &        & I_2
\end{pmatrix} \ \ (n:\mbox{even}), \quad
\widetilde U_n=
\begin{pmatrix}
 I_2 & *   & \cdots & * & * \\
     & I_2 & \ddots & \vdots & \vdots \\
     &     & \ddots & * & * \\
     &     &        & I_2 & * \\
 0   &     &        &     & 1
\end{pmatrix} \ \ (n:\mbox{odd})
\end{equation}
such that
\[
 \widetilde U_n^\top S_n \widetilde U_n = \begin{cases}
 \diag\bigl(\sigma_0 J,\ldots,\sigma_{n/2-1}J\bigr) & (n:\mbox{even}), \\
 \diag\bigl(\sigma_0 J,\ldots,\sigma_{(n-1)/2-1}J,0\bigr) & (n:\mbox{odd}),
 \end{cases}
\qquad J=\JJ.
\]
Here, $\sigma_k$'s are positive because $\pf(S_{2k})>0$ for all $k$.
For example, for
\[
 S_5 = \begin{pmatrix}
 S_{11} & S_{12} & \bm{s}_{13} \\
 S_{21} & S_{22} & \bm{s}_{23} \\
 \bm{s}_{31} & \bm{s}_{32} & 0
 \end{pmatrix} \in \Skew(5)
\]
with $S_{11},S_{22}\in\Skew(2)$, $S_{12}=-S_{21}^\top\in\R^{2\times 2}$, $\bm{s}_{13}=-\bm{s}_{31}^\top$, $\bm{s}_{23}=-\bm{s}_{32}^\top\in\R^{2\times 1}$,
\[
 \widetilde U_5 = \begin{pmatrix}
 I_2 & -S_{11}^{-1}S_{12} & -S_{11}^{-1}\bm{s}_{13} \\
 0   &  I_2               & -S_{22\cdot 1}^{-1}\bm{s}_{23\cdot 1} \\
 0   &  0                 & 1
 \end{pmatrix}
\]
where $S_{22\cdot 1}=S_{22}-S_{21}S_{11}^{-1}S_{12}$ and $\bm{s}_{23\cdot 1}=\bm{s}_{23}-\bm{s}_{21}S_{11}^{-1}\bm{s}_{13}$,
and
\[
  \widetilde U_5^\top S_5 \widetilde U_5
 = \begin{pmatrix}
 S_{11} & 0             & 0 \\
 0      & S_{22\cdot 1} & 0 \\
 0      & 0             & 0
 \end{pmatrix}
 = \begin{pmatrix}
 \sigma_0 J & 0          & 0 \\
 0          & \sigma_1 J & 0 \\
 0          & 0          & 0
 \end{pmatrix}.
\]

Let
\begin{equation}
\label{tildephi}
 (\widetilde\varphi_0(x),\widetilde\varphi_1(x),\ldots,\widetilde\varphi_{n-1}(x)) = (1,x,\ldots,x^{n-1}) \widetilde U_n
\end{equation}
and
\[
 (\widetilde\gamma_0,\ldots,\widetilde\gamma_{n-1}) = (\mu_0,\ldots,\mu_{n-1}) \widetilde U_n.
\]
We will prove later that
\begin{equation}
\label{gamma_i>0}
 \widetilde\gamma_i = \int \widetilde\varphi_i(x) w(x) \dd x > 0 \quad \mbox{for $i$ even}.
\end{equation}
If this is true,
\begin{equation}
\label{unique_UT}
 U_n = \begin{cases}
 \widetilde U_n \diag\Biggl(\begin{pmatrix} 1 & -\widetilde\gamma_{1}/\widetilde\gamma_{0} \\ 0 & 1 \end{pmatrix},\ldots, \begin{pmatrix} 1 & -\widetilde\gamma_{n-1}/\widetilde\gamma_{n-2} \\ 0 & 1 \end{pmatrix}\Biggr) & (n:\mbox{even}), \\[4mm]
 \widetilde U_n \diag\Biggl(\begin{pmatrix} 1 & -\widetilde\gamma_{1}/\widetilde\gamma_{0} \\ 0 & 1 \end{pmatrix},\ldots, \begin{pmatrix} 1 & -\widetilde\gamma_{n-2}/\widetilde\gamma_{n-3} \\ 0 & 1 \end{pmatrix},1\Biggr) & (n:\mbox{odd})
\end{cases}
\end{equation}
is an $n\times n$ upper triangle matrix with unit diagonal elements satisfying (\ref{cond1'}) and (\ref{cond2'}).

The uniqueness of $U_n$ is proven as follows:
Monic polynomials $\varphi_i(x)$'s satisfy the skew-orthogonality (\ref{cond1'}) if and only if $U_n$ defining $\varphi_i(x)$'s by (\ref{varphi}) satisfies
\begin{equation}
\label{USU}
 U_n^\top S_n U_n = \begin{cases}
 \diag\bigl(\sigma_0 J,\ldots,\sigma_{n/2-1}J\bigr) & (n:\mbox{even}), \\
 \diag\bigl(\sigma_0 J,\ldots,\sigma_{(n-1)/2-1}J,0\bigr) & (n:\mbox{odd}),
 \end{cases}
\qquad J=\JJ.
\end{equation}
Let
\[
 \mathrm{UT}(n) =
 \begin{cases}
 \bigl\{\diag\bigl(T(t_1),\ldots,T(t_{n/2})\bigr) \mid t_i\in \R \bigr\} & (n:\mbox{even}), \\
 \bigl\{\diag\bigl(T(t_1),\ldots,T(t_{(n-1)/2}),1\bigr) \mid t_i\in \R \bigr\} & (n:\mbox{odd}),
 \end{cases} \qquad
 T(t)=\begin{pmatrix} 1 & t \\ 0 & 1 \end{pmatrix}
\]
be a subgroup of the upper triangular matrix group.
If a matrix $U_n^0$ satisfies (\ref{USU}), then the matrices in the orbit $\{U_n^0 D \mid D\in \mathrm{UT}(n)\}$ satisfy (\ref{USU}).
This orbit contains a matrix $\widetilde U_n$ of the form (\ref{upper_triangle}), which is uniquely determined by the Gaussian elimination procedure.
Therefore, any matrices $U_n$ satisfying (\ref{USU}) is expressed as $\widetilde U_n D$, $D\in\mathrm{UT}(n)$.
By posing the condition (\ref{cond2'}), the matrix $D$ is determined uniquely as in (\ref{unique_UT}).

Finally, we prove (\ref{gamma_i>0}).
Suppose that $n$ is even.
We will prove that for a random matrix $A_n\in\Sym(n)$ with density $p_n$ in (\ref{pdf}),
\begin{equation}
\label{tildephi-Edet}
 \widetilde\varphi_n(x) = \E[\det(x I_n -A_n)] \quad \mbox{when $n$ is even}.
\end{equation}
If (\ref{tildephi-Edet}) holds, then (\ref{gamma_i>0}) follows from Lemma \ref{lem:even-odd}.

Let
\[
 \widetilde\Phi_n(x) = (\widetilde\varphi_0(x),\ldots,\widetilde\varphi_{n-1}(x))^\top
\]
be the $n\times 1$ column vector consisting of the monic system in (\ref{tildephi}),
and let $L_n$ be an $n\times n$ matrix defined by (\ref{3-term-matrix}), that is,
\[
 x \widetilde\Phi_{n}(x) = \widetilde L_n \widetilde\Phi_{n}(x) + \bm{e}_n \widetilde\varphi_n(x), \quad \bm{e}_n=(0,\ldots,0,1)^\top.
\]
Note that $\widetilde\varphi_n(x)=\det(x I_n-\widetilde L_n)$ by Lemma \ref{lem:monic}.

We use the density function $\q{n}{1}$ in (\ref{pdf-eigen}).
By means of the integral formula in Proposition \ref{prop:pf-int}, we obtain
$\E[\det(z I_n-A_n)]$ as a multiple of
\begin{equation}
\label{pf-even}
\begin{aligned}
& \int_{x_1<\cdots<x_n} \prod_{i=1}^n (z-x_i)w(x_i) \prod_{i<j} (x_j-x_i) \prod_{i=1}^n \dd x_i \\
&= \int_{x_1<\cdots<x_n} \prod_{i=1}^n (z-x_i)w(x_i) \prod_{i<j} \bigl(\widetilde\varphi_j(x)-\widetilde\varphi_i(x)\bigr) \prod_{i=1}^n \dd x_i \\
&= \pf\biggl( \biggl(\int_{x<y}-\int_{x>y}\biggr) (z-x)(z-y)\widetilde\varphi_i(x)\widetilde\varphi_j(y) w(x) w(y) \dd x \dd y \biggr)_{0\le i,j\le n-1} \\
&= \pf\biggl( \biggl(\int_{x<y}-\int_{x>y}\biggr) (z-x)(z-y) \widetilde\Phi_n(x) \widetilde\Phi_n(y)^\top w(x) w(y) \dd x \dd y \biggr)_{n\times n} \\
&= \pf\bigl((I_n z-\widetilde L_n) G (I_n z-\widetilde L_n)^\top - (I_n z-\widetilde L_n) \bm{g} \bm{e}_n^\top + \bm{e}_n \bm{g}^\top (I_n z-\widetilde L_n)^\top\bigr),
\end{aligned}
\end{equation}
where
\begin{align*}
 G =& \biggl(\int_{x<y}-\int_{x>y}\biggr) \widetilde\Phi_{n}(x) \widetilde\Phi_{n}(y)^\top w(x)w(y) \dd x \dd y \\
 =& \diag(\sigma_0 J,\ldots,\sigma_{n/2-1}J), \quad J=\JJ,
\end{align*}
and
\[
 \bm{g} = \biggl(\int_{x<y}-\int_{x>y}\biggr) \widetilde\Phi_{n}(x) \widetilde\varphi_{n}(y) w(x) w(y) \dd x \dd y = 0
\]
because $\widetilde\varphi_i(x)$ satisfies the skew-orthogonality (\ref{cond1}).
Hence, (\ref{pf-even}) is $\det(I_n z-\widetilde L_n) \pf(G) = \widetilde\varphi_n(z) \pf(G)$,
and by checking the coefficient of $z^n$, (\ref{tildephi-Edet}) follows.
\end{proof}

\begin{proof}[Proof of Lemma \ref{lem:norm-sym}]

We apply Proposition \ref{prop:pf-int} with $g(\lambda)=w(\lambda)$ and $\varphi_i(\lambda)$ as the monic skew-orthogonal polynomial with respect to the weight function $w(\lambda)$.
Let $U_n$ be the matrix that defined in Lemma \ref{lem:USU}.
When $n$ is even,
\[
\begin{aligned}
  1/\d{n}{1} = \pf(S) = \pf(U_n^\top S U_n) =& \pf\left(
\diag\Biggl(
 \begin{pmatrix} 0 & \sigma_{0} \\ -\sigma_{0} & 0 \end{pmatrix},\ldots,
 \begin{pmatrix} 0 & \sigma_{n/2-1} \\ -\sigma_{n/2-1} & 0 \end{pmatrix}\right)\Biggr) \\
 =& \prod_{k=0}^{n/2-1}\sigma_k.
\end{aligned}
\]
When $n$ is odd,
\begin{align*}
 1/\d{n}{1} = \pf\begin{pmatrix} S & \bm{s} \\ -\bm{s}^\top & 0 \end{pmatrix}
=& \pf\left(\begin{pmatrix}U_n & 0 \\ 0 & 1\end{pmatrix}^\top\begin{pmatrix} S & \bm{s} \\ -\bm{s}^\top & 0 \end{pmatrix} \begin{pmatrix}U_n & 0 \\ 0 & 1\end{pmatrix}\right) \\[1mm]
=&
\pf\left(\begin{array}{cccccccc|c}
\cline{1-2}
\multicolumn{1}{|c}{} & \multicolumn{1}{c|}{\sigma_0} & & & & & & & \gamma_0 \\
\multicolumn{1}{|c}{-\sigma_0} & \multicolumn{1}{c|}{} & & & & & & & 0 \\ \cline{1-4}
 & & \multicolumn{1}{|c}{} & \multicolumn{1}{c|}{\sigma_1} & & & & & \gamma_2 \\
 & & \multicolumn{1}{|c}{-\sigma_1} & \multicolumn{1}{c|}{} & & & & & 0 \\ \cline{3-4}
 & & & & \ddots & & & & \vdots \\ \cline{6-7}
 & & & & & \multicolumn{1}{|c}{} & \multicolumn{1}{c|}{\sigma_{\frac{n-3}{2}}} & & \gamma_{n-3} \\
 & & & & & \multicolumn{1}{|c}{-\sigma_{\frac{n-3}{2}}} & \multicolumn{1}{c|}{} & & 0 \\ \cline{6-7}
 & & & & & & & 0 & \gamma_{n-1} \\ \hline
 -\gamma_0 & 0 & -\gamma_2 & 0 & \cdots & -\gamma_{n-3} & 0 & -\gamma_{n-1} & 0
\end{array}\right) \\
=& \prod_{k=0}^{(n-3)/2}\sigma_k \times \gamma_{n-1}.
\end{align*}
\end{proof}

\begin{proof}[Proof of Lemma \ref{lem:Edet}]

Let $\widetilde\varphi_i$ and $\widetilde\gamma_i$ be defined in the proof of Lemma \ref{lem:USU}.
Note first that for $i$ even,
\begin{equation}
\label{even_i}
 (\gamma_i,0) =
 (\widetilde\gamma_i,\widetilde\gamma_{i+1})
 \begin{pmatrix} 1 & -\widetilde\gamma_{i+1}/\widetilde\gamma_i \\ 0 & 1 \end{pmatrix},
\quad
 (\varphi_i(x),\varphi_{i+1}(x)) =
 (\widetilde\varphi_i(x),\widetilde\varphi_{i+1}(x))
 \begin{pmatrix} 1 & -\widetilde\gamma_{i+1}/\widetilde\gamma_i \\ 0 & 1 \end{pmatrix},
\end{equation}
hence, $\gamma_i=\widetilde\gamma_i$ and $\varphi_i(x)=\widetilde\varphi_i(x)$ for $i$ even.

Suppose that $n$ is even.
Then,
\begin{align*}
 \E[\det(z I_n-A_n)]
 =& \widetilde\varphi_n(z) \quad (\mbox{proof of Lemma \ref{lem:USU}}) \\
 =& \varphi_n(z) \quad (\mbox{eq.\ (\ref{even_i})}) \\
 =& \widehat\varphi_n(z) \quad (\mbox{eq.\ (\ref{hatvarphi})}).
\end{align*}

In the following, suppose that $n$ is odd, and prove that $\E[\det(z I_n-A_n)] = \widehat\varphi_n(z)$.
The transformation from $\Phi_n(x)=(\varphi_0(x),\ldots,\varphi_{n-1}(x))^\top$ to $\widehat\Phi_n(x)=(\widehat\varphi_0(x),\ldots,\widehat\varphi_{n-1}(x))^\top$ is
$\widehat\Phi_n(x) = B_n^\top\Phi_n(x)$, where
\[
B_n=
\begin{pmatrix}
 1 & 0 & 0 & 0 & & 0 & 0 & 0 & 0 \\
   & 1 & 0 & \gamma_3/\gamma_1 & & \gamma_{n-4}/\gamma_1 & 0 & \gamma_{n-2}/\gamma_1 & 0 \\
   &   & 1 & 0 &        & 0      & 0      & 0      & 0 \\
   &   &   & 1 & & \gamma_{n-4}/\gamma_3 & 0 & \gamma_{n-2}/\gamma_3 & 0 \\
   &   &   &   & \ddots & \vdots & \vdots & \vdots & \vdots \\
   &   &   &   &        & 1      & 0      & \gamma_{n-2}/\gamma_{n-4} & 0 \\
   &   &   &   &        &        & 1      & 0      & 0 \\
   &   &   &   &        &        &        & 1      & 0 \\
   &   &   &   &        &        &        &        & 1
\end{pmatrix}_{n\times n}.
\]
Then, we have
\begin{equation}
\label{intint}
\begin{aligned}
 \biggl(\int_{x<y}-\int_{x>y}\biggr) & \widehat\varphi_i(x) \widehat\varphi_j(y) w(x) w(y) \dd x \dd y \\
 &= \left(B_n^\top \diag\Biggl(\sigma_0 \JJ,\ldots,\sigma_{(n-1)/2-1}\JJ,0 \Biggr) B_n\right)_{ij} \\
 &= \begin{cases}
 \gamma_i \gamma_j  & (i<j,\,i:\mbox{even},\,j:\mbox{odd}), \\
 -\gamma_i \gamma_j & (i>j,\,i:\mbox{odd},\,j:\mbox{even}), \\
 0                  & (\mbox{otherwise}),
 \end{cases}
\end{aligned}
\end{equation}
and
\begin{equation}
\label{int}
\begin{aligned}
 \int \widehat\varphi_i(x) w(x) \dd x
 &= \bigl((\gamma_0,0,\gamma_2,0,\ldots,0,\gamma_{n-1}) B_n\bigr)_i \\
 &= \begin{cases}
 \gamma_i & (i:\mbox{even}), \\
 0        & (i:\mbox{odd}).
 \end{cases}
\end{aligned}
\end{equation}
Note that $\widehat\varphi_i(x)$ and $\gamma_i$ are defined independently of $n$.
Hence, (\ref{intint}) and (\ref{int}) hold for any $i,j$.

Define a matrix $\widehat L_n$ by
\[
 x\widehat\Phi_n(x) = \widehat L_n \widehat\Phi_n(x) + \bm{e}_n \widehat\varphi_n(x), \quad \bm{e}_n=(0,\ldots,0,1)^\top
\]
as before.
$\E[\det(z I_n-A_n)]$ is proportional to
\begin{equation}
\label{pf-odd}
\begin{aligned}
& \int_{x_1<\cdots<x_n} \prod_{i=1}^n (z-x_i) w(x_i) \prod_{i<j} (x_j-x_i) \prod_{i=1}^n \dd x_i \\
&= \pf\left(\begin{matrix}
 \Bigl(\bigl(\int_{x<y}-\int_{x>y}\bigr) (z-x)(z-y) \varphi_i(x) \varphi_j(y) w(x) w(y) \dd x \dd y \Bigr)_{0\le i,j\le n-1} & * \\
 -\Bigl(\int (z-y) \varphi_j(y) w(y) \dd y \Bigr)_{0\le j\le n-1} & 0
 \end{matrix}\right) \\
&=
 \pf\left(\begin{matrix}
 \Bigl(\bigl(\int_{x<y}-\int_{x>y}\bigr) (z-x)(z-y) \widehat\Phi_{n}(x) \widehat\Phi_{n}(y)^\top w(x) w(y) \dd x \dd y \Bigr)_{n\times n} &
* \\
 -\Bigl(\int (z-y) \widehat\Phi_{n}(y)^\top w(y) \dd y \Bigr)_{1\times n} & 0
 \end{matrix}
\right).
\end{aligned}
\end{equation}
Here,
\begin{align*}
& \biggl(\int_{x<y}-\int_{x>y}\biggr) (z-x)(z-y) \widehat\Phi_{n}(x) \widehat\Phi_{n}(y)^\top w(x) w(y) \dd x \dd y \\
&= (I_n z-\widehat L_n) \widehat G (I_n z-\widehat L_n)^\top - (I_n z-\widehat L_n) \widehat{\bm{g}} \bm{e}_n^\top + \bm{e}_n \widehat{\bm{g}}^\top (I_n z-\widehat L_n)^\top,
\end{align*}
where
\[
 \widehat G = \biggl(\int_{x<y}-\int_{x>y}\biggr) \widehat\Phi_{n}(x) \widehat\Phi_{n}(y)^\top w(x) w(y) \dd x \dd y
\]
and
\[
 \widehat{\bm{g}} = \biggl(\int_{x<y}-\int_{x>y}\biggr) \widehat\Phi_{n}(x) \widehat\varphi_{n}(y) w(x) w(y) \dd x \dd y = \gamma_n \bm{\gamma}, \quad
 \bm{\gamma} = (\gamma_0,0,\gamma_2,0,\ldots,\gamma_{n-1})^\top,
\]
by (\ref{intint}).
Moreover,
\[
 \int (z-y) \widehat\Phi_{n}(y) w(y) \dd y = (I_n z-\widehat L_n)\int\widehat\Phi_n(y)w(y) \dd y = (I_n z-\widehat L_n) \bm{\gamma}
\]
by (\ref{int}).
Hence, (\ref{pf-odd}) is
\begin{align*}
& \pf\begin{pmatrix}
 (I_n z -\widehat L_n) \widehat G (I_n z -\widehat L_n)^\top - \gamma_n (I_n z -\widehat L_n) \bm{\gamma} \bm{e}_n^\top + \gamma_n \bm{e}_n \bm{\gamma}^\top (I_n z -\widehat L_n)^\top & (I_n z -\widehat L_n) \bm{\gamma} \\
 -\bm{\gamma}^\top (I_n z -\widehat L_n)^\top & 0
 \end{pmatrix} \\
&= \pf\begin{pmatrix}
 (I_n z -\widehat L_n) \widehat G (I_n z -\widehat L_n)^\top & \bm{\gamma} (I_n z -\widehat L_n) \\
 -\bm{\gamma}^\top (I_n z -\widehat L_n)^\top & 0
 \end{pmatrix} \\
&= \pf\Biggl[\begin{pmatrix} I_n z -\widehat L_n & 0 \\ 0 & 1 \end{pmatrix}
 \begin{pmatrix} \widehat G & \bm{\gamma} \\ -\bm{\gamma}^\top & 0 \end{pmatrix}
\begin{pmatrix} I_n z -\widehat L_n & 0 \\ 0 & 1 \end{pmatrix}^\top\Biggr] \\
&= \det(I_n z -\widehat L_n)\pf\begin{pmatrix} \widehat G & \bm{\gamma} \\ -\bm{\gamma} & 0 \end{pmatrix}
= \widehat\varphi_n(z)\pf\begin{pmatrix} \widehat G & \bm{\gamma} \\ -\bm{\gamma} & 0 \end{pmatrix},
\end{align*}
where the Gaussian elimination (\ref{elimination}) is used.
This means $\E[\det(z I_n -A_n)]=\widehat\varphi_n(z)$ for an odd $n$.
The proof is completed.
\end{proof}

\subsection{Proof of Theorem \ref{thm:ec-limit}}

\begin{proof}[Proof of Theorem \ref{thm:ec-limit}]

We state the proof mainly in the Gaussian case.
Let
\begin{equation}
\label{psi}
 \psi_n(x) = \frac{\sigma_n}{\gamma_n} w(x) \bar H_{n}(x)
\end{equation}
and $\psi'_n(x)$ be its derivative.
Proposition \ref{prop:johnstone} claims that
\begin{equation}
\label{psi-limit}
 \psi_n(\mu_n + \sigma_n s) \to \frac{1}{2}\Ai(s), \qquad \sigma_n \psi_n'(\mu_n + \sigma_n s) \to \frac{1}{2}\Ai'(s).
\end{equation}

{\rm (I)} Real symmetric case (GOE).

According to Proposition \ref{prop:johnstone} (i) showing the uniform convergence, the expected Euler characteristic with the threshold $x=\mu_{n-1}+\sigma_{n-1} s$ converges as
\[
 \frac{1}{\gamma_{n-1}}\int_x^\infty \bar w(x) H_{n-1}(x) \dd x
 = \int_s^\infty \psi_{n-1}(\mu_{n-1}+\sigma_{n-1}s) \dd s \to
 \frac{1}{2}\int_s^\infty \Ai(s) \dd s.
\]

{\rm (II)} Hermitian case (GUE).

According to Proposition \ref{prop:johnstone} (i), the expected Euler characteristic with the threshold $x$ is
\begin{equation}
\label{tmp}
 \frac{1}{h_{n-1}}\int_x^\infty w(x)^2 \big(\bar H_{n-1}(x) \bar H_{n}'(x)-\bar H_{n}(x) \bar H_{n-1}'(x) \bigr) \dd x,
\end{equation}
where $w(x)=e^{-x^2/2}$, $h_{n-1}=2^{-n+1}\sqrt{\pi}\Gamma(n)$ (Table \ref{tab:ortho_poly}).
Substituting
$\bar H_n'(x)= n \bar H_{n-1}(x)$ and
$\bar H_n(x)=x\bar H_{n-1}(x)-(1/2)\bar H_{n-1}'(x)$,
(\ref{tmp}) is expressed in terms of $\bar H_{n-1}(x)$ and $\bar H_{n-1}'(x)$.
Then, by substituting
\[
 \bar H_{n-1}(x)=\frac{\psi_{n-1}(x)}{w(x)}, \qquad
 \bar H_{n-1}'(x)=\frac{\psi_{n-1}'(x)-w'(x)\psi_{n-1}(x)/w(x)}{w(x)},
\]
 (\ref{tmp}) is expressed in terms of $\psi_{n-1}(x)$ in (\ref{psi}) and $\psi_{n-1}'(x)$.
Substituting $x=\mu_{n-1}+\sigma_{n-1} s$, and taking the limit using (\ref{psi-limit}), we obtain the result.

{\rm (III) Other cases.}

The proofs for the real symmetric matrices are the same as the Gaussian case.

For the complex Wishart matrix, we use the formulas:
\[
 (\bar L_n^{(\alpha)}{}(x))' = n \bar L_{n-1}^{(\alpha+1)}(x)
\]
and the three-term relation
\[
 \bar L_{n}^{(\alpha)}(x) = (x-\alpha-1)\bar L_{n-1}^{(\alpha+1)}(x) -x \bigl(\bar L_{n-1}^{(\alpha+1)}(x)\bigr)'.
\]
Then, the expected Euler characteristic formula is expressed in terms of
\[
 \psi_{n-1}^{(\alpha)}(x) = \frac{\sigma_{n-1}}{\gamma_{n-1}} w^{(\alpha)}(x) \bar L_{n}^{(\alpha)}(x)
\]
($\alpha=\alpha,\alpha+1$) and their derivatives.
By arranging the terms and taking the limit, we obtain the same result.
The complex beta matrix case is the same.
\end{proof}

\appendix

\section{Exponential asymptotic structure of Painlev\'e II}
\label{sec:painleve}

Let $\Ai(x)$ and $\Bi(x)$ be the Airy functions of the first  and second kinds, respectively.
Let $q(x)$ be the solution to Painlev\'e II:
\begin{equation}
\label{Painleve}
 0 = q''-2 q^3-x q, \quad q(x)\sim \Ai(x)\ \ (x\to\infty).
\end{equation}

The Airy functions are independent solutions of $0=f''(x)-x f(x)$.
Let
\[
 W(x,s)=\Ai(x)\Bi(s)-\Bi(x)\Ai(s)
\]
be Wronskian.
It holds the identity
\begin{equation}
\label{q-identity}
 q(x) = 2\pi\int_x^\infty W(x,s) q(s)^3 \dd s + \Ai(x).
\end{equation}
\begin{theorem}
\label{thm:exp_asympt}
Assume that the solution $q(x)$ of (\ref{Painleve}) uniquely exists.
Then, $q(x)$ has an exponential asymptotic structure
\begin{equation}
\label{q(x)}
 q(x) = q_0(x)+q_1(x)+\cdots+q_M(x)+r_M(x),
\end{equation}
where
\begin{equation}
\begin{aligned}
 q_0(x) =& \Ai(x), \\
 q_m(x) =& 2\pi\int_x^\infty W(x,s) f_m(s) \dd s, \quad f_m(x)=\sum_{\substack{i_1+i_2+i_3=m-1, \\ i_1,i_2,i_3\ge 0}} q_{i1}(x) q_{i2}(x) q_{i3}(x),
\end{aligned}
\label{recursive}
\end{equation}
and as $x\to\infty$,
\begin{equation}
\label{qm1}
 q_m(x) \sim \frac{1}{2^{4m+1}\pi^{(2m+1)/2}x^{(6m+1)/4}}e^{-\frac{4m+2}{3}x^{3/2}},
\end{equation}
\begin{equation}
 r_M(x) = o\bigl(q_{M}(x)\bigr).
\label{rM}
\end{equation}
\end{theorem}

\begin{remark}
The first two terms in $(\ref{q(x)})$ are
\begin{align*}
& q_0(x) = \Ai(x) \sim \frac{1}{2\pi^{1/2}x^{1/4}}e^{-\frac{2}{3}x^{3/2}}, \\
& q_1(x) =
 2\pi \int_x^\infty W(x,s) \Ai(s)^3 \, \dd s
 \sim \frac{1}{2^5\pi^{3/2}x^{7/4}}e^{-2 x^{3/2}}
\end{align*}
as $x\to\infty$
(cf.\ \cite[eqs.\ (10.4.59) and (10.4.63)]{abramowitz-stegun:1970}).
\end{remark}

\begin{lemma}
\label{lm:order}
Suppose that $\rho(x)>0$.
For $f(x)=o(\rho(x))$ as $x\to\infty$,
\[
 \int_x^\infty W(x,s) f(s) \dd s = o\biggl(\int_x^\infty W(x,s) \rho(s) \dd s\biggr).
\]
For $f(x)\sim \rho(x)$ as $x\to\infty$,
\[
 \int_x^\infty W(x,s) f(s) \dd s \sim \int_x^\infty W(x,s) \rho(s) \dd s.
\]
\end{lemma}
\begin{proof}
For all sufficiently large $x$ and $x<s$, $\Ai(x)>\Ai(s)$, $\Bi(x)<\Bi(s)$, and hence $W(x,s)>0$.
For the former case, for $\forall\varepsilon>0$, $-\varepsilon\rho(x)\le f(x) \le \varepsilon\rho(x)$ for sufficiently large $x$, and
\[
 -\varepsilon \int_x^\infty W(x,s) \rho(s) \dd s \le \int_x^\infty W(x,s) f(s) \dd s \le \varepsilon \int_x^\infty W(x,s) \rho(s) \dd s
\]
follows.
The same for the latter case.
\end{proof}

\begin{lemma}
\label{lm:tail}
For $a\ge 2$, $b\ge 3/4$, as $x\to\infty$,
\[
 2\pi\int_x^\infty W(x,s) s^{-b} e^{-a s^{3/2}} \dd s \sim \frac{8}{(3a+2)(3a-2)} x^{-(b+1)}e^{-a x^{3/2}}.
\]
\end{lemma}

\begin{proof}[Proof of Theorem \ref{thm:exp_asympt}]
Let $q_m$ ($0\le m\le M$) be defined by (\ref{recursive}) recursively, and let
\begin{equation}
\label{rM_def}
 r_M(x) = q(x)-\sum_{m=0}^M q_m(x).
\end{equation}
We will prove (\ref{qm1}) and (\ref{rM}).

(\ref{qm1}) is proved by induction.
The case $m=0$ holds by the assumption of the existence of the solution $q(x)$.
Suppose that (\ref{qm1}) holds for $m=0,\ldots,m-1$.
For $i_1+i_2+i_3=m-1$,
\[
 q_{i_1}(x)q_{i_2}(x)q_{i_3}(x) \sim
 \frac{1}{2^{4(m-1)+3}\pi^{(2(m-1)+3)/2}x^{(6(m-1)+3)/4}}e^{-\frac{4(m-1)+6}{3}x^{3/2}},
\]
and
\[
 \sum_{\substack{i_1+i_2+i_3=m-1, \\ i_1,i_2,i_3\ge 0}} 1 = \binom{m+1}{2},
\]
hence
\begin{equation*}
 f_m(x) \sim \binom{m+1}{2} \frac{1}{2^{4m-1}\pi^{(2m+1)/2}x^{(6m-3)/4}}e^{-\frac{4m+2}{3}x^{3/2}}
\end{equation*}
and by Lemmas \ref{lm:order} and \ref{lm:tail},
\[
 q_m(x) = 2\pi\int_x^\infty W(x,s) f_m(s) \dd s \sim
 \frac{1}{2^{4m+1}\pi^{(2m+1)/2}x^{(6m+1)/4}}e^{-\frac{4m+2}{3}x^{3/2}}.
\]

To prove (\ref{rM}), we start from the identity (\ref{q-identity}).
Substituting
\begin{align*}
 q^3
 =& \biggl(\sum_{m=0}^M q_m + r_M\biggr)^3 \\
 =& \biggl(\sum_{m=0}^M q_m\biggr)^3 + 3\biggl(\sum_{m=0}^M q_m\biggr)^2 r_M + 3\biggl(\sum_{m=0}^M q_m\biggr) r_M^2 + r_M^3 \\
 =& \sum_{m=1}^M f_m + \sum_{m=M+1}^{3M+1} f_m + 3\biggl(\sum_{m=0}^M q_m\biggr)^2 r_M + 3\biggl(\sum_{m=0}^M q_m\biggr) r_M^2 + r_M^3
\end{align*}
into (\ref{q-identity}),
and noting the definitions of $r_M$ in (\ref{rM_def}) and $q_m$ in (\ref{recursive}), we have
\begin{equation}
\label{rM1}
\begin{aligned}
 r_M(x)
 =& q(x) - \sum_{m=1}^M q_m(x) -\Ai(x) \\
 =& 2\pi\int_x^\infty W(x,s) q(s)^3 \dd s - \sum_{m=1}^M q_m(x) \\
 =& 2\pi\int_x^\infty W(x,s) \times\Biggl\{ \sum_{m=M+1}^{3M+1} f_m(s) \\
  & \qquad + 3\biggl(\sum_{m=0}^M q_m(s)\biggr)^2 r_M(s) + 3\biggl(\sum_{m=0}^M q_m(s)\biggr) r_M(s)^2 + r_M(s)^3 \Biggr\} \dd s.
\end{aligned}
\end{equation}

Here, $q(x)$ is the solution with the boundary condition $q(x)=\Ai(x)(1+o(1))$, and $q_0(x)=\Ai(x)$, and
noting the order $q_m(x)=o(\Ai(x))$ for $m\ge 1$ by (\ref{qm1}),
we have $r_M(x)=o(\Ai(x))=o\bigl(x^{-1/4}e^{-2/3 x^{3/2}}\bigr)$.
Hence, in the inside of $\{\cdot\}$ in (\ref{rM1}),
the dominant term is $q_0(s)^2 r_M(s)$ (when $M>0$), which is $o\bigl(s^{-3/4}e^{-6/3 s^{3/2}}\bigr)$.
Therefore, by Lemmas \ref{lm:order} and \ref{lm:tail}, we have $r_M(x) =o\bigl(x^{-7/4}e^{-6/3 x^{3/2}}\bigr)$.

Using the new order of $r_M(x)$, we evaluate the inside of $\{\cdot\}$ in (\ref{rM1}) again.
The dominant term (when $M>1$) is $q_0(s)^2 r_M(s)=o\bigl(s^{-9/4}e^{-10/3 s^{3/2}}\bigr)$ and by Lemmas \ref{lm:order} and \ref{lm:tail}, we have $r_M(x)=o\bigl(x^{-13/4}e^{-10/3 x^{3/2}}\bigr)$.
This procedure terminates when $f_{M+1}(s)$ dominates $q_0(s)^2 r_M(s)$.
In this last stage,
\[
 r_M(x)=2\pi\int_x^\infty W(x,s) \{ f_{M+1}(s) + \cdots \} \dd s = O(q_{M+1}(x)) = o(q_{M}(x)).
\]
\end{proof}

\bibliographystyle{amsalpha}
\bibliography{ec_infinity-bib.bib}

\providecommand{\bysame}{\leavevmode\hbox to3em{\hrulefill}\thinspace}
\providecommand{\MR}{\relax\ifhmode\unskip\space\fi MR }
\providecommand{\MRhref}[2]{%
  \href{http://www.ams.org/mathscinet-getitem?mr=#1}{#2}
}
\providecommand{\href}[2]{#2}
\begin{thebibliography}{AFNvM00}

\bibitem[Adl81]{adler:1981}
Robert~J. Adler, \emph{The geometry of random fields}, John Wiley \& Sons,
  Ltd., Chichester, 1981, Wiley Series in Probability and Mathematical
  Statistics.

\bibitem[AFNvM00]{adler-etal:2000}
M.~Adler, P.~J. Forrester, T.~Nagao, and P.~van Moerbeke, \emph{Classical skew
  orthogonal polynomials and random matrices}, J. Statist. Phys. \textbf{99}
  (2000), no.~1-2, 141--170. \MR{1762659}

\bibitem[AS64]{abramowitz-stegun:1970}
Milton Abramowitz and Irene~A. Stegun, \emph{Handbook of mathematical functions
  with formulas, graphs, and mathematical tables}, National Bureau of Standards
  Applied Mathematics Series, vol. No. 55, U. S. Government Printing Office,
  Washington, DC, 1964, For sale by the Superintendent of Documents.
  \MR{167642}

\bibitem[AT07]{adler-taylor:2007}
Robert~J. Adler and Jonathan~E. Taylor, \emph{Random fields and geometry},
  Springer Monographs in Mathematics, Springer, New York, 2007.

\bibitem[Bal02]{ball:2002}
James~S. Ball, \emph{Half-range generalized {H}ermite polynomials and the
  related {G}aussian quadratures}, SIAM J. Numer. Anal. \textbf{40} (2002),
  no.~6, 2311--2317 (2003). \MR{1974187}

\bibitem[BR01]{baik-rains:2001}
Jinho Baik and Eric~M. Rains, \emph{Algebraic aspects of increasing
  subsequences}, Duke Math. J. \textbf{109} (2001), no.~1, 1--65. \MR{1844203}

\bibitem[DE02]{dumitriu-edelman:2002}
Ioana Dumitriu and Alan Edelman, \emph{Matrix models for beta ensembles}, J.
  Math. Phys. \textbf{43} (2002), no.~11, 5830--5847. \MR{1936554}

\bibitem[Dei99]{deift:1999}
P.~A. Deift, \emph{Orthogonal polynomials and random matrices: a
  {R}iemann-{H}ilbert approach}, Courant Lecture Notes in Mathematics, vol.~3,
  New York University, Courant Institute of Mathematical Sciences, New York;
  American Mathematical Society, Providence, RI, 1999. \MR{1677884}

\bibitem[Gho09]{ghosh:2009}
Saugata Ghosh, \emph{Skew-orthogonal polynomials and random matrix theory}, CRM
  Monograph Series, vol.~28, American Mathematical Society, Providence, RI,
  2009. \MR{2560290}

\bibitem[Hab81]{haberman:1981}
Shelby~J. Haberman, \emph{Tests for independence in two-way contingency tables
  based on canonical correlation and on linear-by-linear interaction}, Ann.
  Statist. \textbf{9} (1981), no.~6, 1178--1186. \MR{630101}

\bibitem[HT68]{hanumara-thompson:1968}
R.~Choudary Hanumara and W.~A. Thompson, Jr., \emph{Percentage points of the
  extreme roots of a {W}ishart matrix}, Biometrika \textbf{55} (1968),
  505--512. \MR{237044}

\bibitem[IW95]{ishikawa-wakayama:1995}
Masao Ishikawa and Masato Wakayama, \emph{Minor summation formula of
  {P}faffians}, Linear and Multilinear Algebra \textbf{39} (1995), no.~3,
  285--305. \MR{1365449}

\bibitem[JM12]{johnstone-ma:2012}
Iain~M. Johnstone and Zongming Ma, \emph{Fast approach to the {T}racy-{W}idom
  law at the edge of {GOE} and {GUE}}, Ann. Appl. Probab. \textbf{22} (2012),
  no.~5, 1962--1988. \MR{3025686}

\bibitem[Joh01]{johnstone:2001}
Iain~M. Johnstone, \emph{On the distribution of the largest eigenvalue in
  principal components analysis}, Ann. Statist. \textbf{29} (2001), no.~2,
  295--327. \MR{1863961}

\bibitem[Joh08]{johnstone:2008}
\bysame, \emph{Multivariate analysis and {J}acobi ensembles: largest
  eigenvalue, {T}racy-{W}idom limits and rates of convergence}, Ann. Statist.
  \textbf{36} (2008), no.~6, 2638--2716. \MR{2485010}

\bibitem[Kat14]{kateri:2014}
Maria Kateri, \emph{Contingency table analysis}, Statistics for Industry and
  Technology, Birkh\"{a}user/Springer, New York, 2014, Methods and
  implementation using R. \MR{3290014}

\bibitem[KT01]{kuriki-takemura:2001}
Satoshi Kuriki and Akimichi Takemura, \emph{Tail probabilities of the maxima of
  multilinear forms and their applications}, Ann. Statist. \textbf{29} (2001),
  no.~2, 328--371. \MR{1863962}

\bibitem[KT08]{kuriki-takemura:2008}
\bysame, \emph{Euler characteristic heuristic for approximating the
  distribution of the largest eigenvalue of an orthogonally invariant random
  matrix}, J. Statist. Plann. Inference \textbf{138} (2008), no.~11,
  3357--3378. \MR{2450081}

\bibitem[KTT22]{kuriki-takemura-taylor:2022}
Satoshi Kuriki, Akimichi Takemura, and Jonathan~E. Taylor, \emph{The
  volume-of-tube method for {G}aussian random fields with inhomogeneous
  variance}, J. Multivariate Anal. \textbf{188} (2022), Paper No. 104819, 23.
  \MR{4353844}

\bibitem[Kur93]{kuriki:1993}
Satoshi Kuriki, \emph{One-sided test for the equality of two covariance
  matrices}, Ann. Statist. \textbf{21} (1993), no.~3, 1379--1384. \MR{1241270}

\bibitem[Meh04]{mehta:2004}
Madan~Lal Mehta, \emph{Random matrices}, third ed., Pure and Applied
  Mathematics (Amsterdam), vol. 142, Elsevier/Academic Press, Amsterdam, 2004.
  \MR{2129906}

\bibitem[Mui82]{muirhead:1982}
Robb~J. Muirhead, \emph{Aspects of multivariate statistical theory}, John Wiley
  \& Sons, Inc., New York, 1982, Wiley Series in Probability and Mathematical
  Statistics. \MR{652932}

\bibitem[Nag07]{nagao:2007}
Taro Nagao, \emph{Pfaffian expressions for random matrix correlation
  functions}, J. Stat. Phys. \textbf{129} (2007), no.~5-6, 1137--1158.
  \MR{2363392}

\bibitem[NF95]{nagao-forrester:1995}
Taro Nagao and Peter~J. Forrester, \emph{Asymptotic correlations at the
  spectrum edge of random matrices}, Nuclear Phys. B \textbf{435} (1995),
  no.~3, 401--420. \MR{1313625}

\bibitem[NW91]{nagao-wadati:1991}
Taro Nagao and Miki Wadati, \emph{Correlation functions of random matrix
  ensembles related to classical orthogonal polynomials}, J. Phys. Soc. Japan
  \textbf{60} (1991), no.~10, 3298--3322. \MR{1142971}

\bibitem[Roy53]{roy:1953}
S.~N. Roy, \emph{On a heuristic method of test construction and its use in
  multivariate analysis}, Ann. Math. Statistics \textbf{24} (1953), 220--238.
  \MR{57519}

\bibitem[Sze75]{szego:1975}
G\'{a}bor Szeg\H{o}, \emph{Orthogonal polynomials}, fourth ed., American
  Mathematical Society Colloquium Publications, Vol. XXIII, American
  Mathematical Society, Providence, R.I., 1975. \MR{0372517}

\bibitem[TJKZ20]{takayama-etal:2020}
Nobuki Takayama, Lin Jiu, Satoshi Kuriki, and Yi~Zhang, \emph{Computation of
  the expected {E}uler characteristic for the largest eigenvalue of a real
  non-central {W}ishart matrix}, J. Multivariate Anal. \textbf{179} (2020),
  104642, 18. \MR{4114967}

\bibitem[TK02]{takemura-kuriki:2002}
Akimichi Takemura and Satoshi Kuriki, \emph{On the equivalence of the tube and
  {E}uler characteristic methods for the distribution of the maximum of
  {G}aussian fields over piecewise smooth domains}, Ann. Appl. Probab.
  \textbf{12} (2002), no.~2, 768--796.

\bibitem[TW94]{tracy-widom:1994}
Craig~A. Tracy and Harold Widom, \emph{Level-spacing distributions and the
  {A}iry kernel}, Comm. Math. Phys. \textbf{159} (1994), no.~1, 151--174.
  \MR{1257246}

\bibitem[TW96]{tracy-widom:1996}
\bysame, \emph{On orthogonal and symplectic matrix ensembles}, Comm. Math.
  Phys. \textbf{177} (1996), no.~3, 727--754. \MR{1385083}

\bibitem[Wor95]{worsley:1995}
Keith~J. Worsley, \emph{Boundary corrections for the expected {E}uler
  characteristic of excursion sets of random fields, with an application to
  astrophysics}, Adv. in Appl. Probab. \textbf{27} (1995), no.~4, 943--959.
  \MR{1358902}

\end{thebibliography}

\end{document}